\documentclass[a4paper,twoside]{amsart}
\pdfoutput=1

\usepackage[T1]{fontenc}
\usepackage[utf8]{inputenc}
\usepackage{lmodern}
\usepackage{enumerate}

\usepackage{amsmath}
\usepackage{amsthm}
\usepackage{amssymb}
\usepackage{amscd}
\usepackage{mathrsfs}
\usepackage{bbold}

\usepackage{tikz}
\usetikzlibrary{arrows,backgrounds,matrix,decorations.pathreplacing,shapes}

\usepackage{scrpage2}

\usepackage{thumbpdf}
\usepackage{hypbmsec}
\usepackage{url}
\usepackage[colorlinks,linkcolor=blue,bookmarksopen=true,bookmarksopenlevel=4,final]{hyperref}
\usepackage{aliascnt}
\usepackage{cleveref}

\setheadsepline{0.4pt}
\automark[subsection]{section}
\ihead{A Simplicial Calculus for Local Intersection Numbers}

\makeatletter
\newtheorem*{rep@theorem}{\rep@title}
\newcommand{\newreptheorem}[2]{%
\newenvironment{rep#1}[1]{%
 \def\rep@title{#2 \ref*{##1}}%
 \begin{rep@theorem}}%
 {\end{rep@theorem}}}
\makeatother

\newtheorem{satz}{Theorem}[section]
\newreptheorem{satz}{Theorem}

\newaliascnt{lem}{satz}
\newtheorem{lem}[lem]{Lemma}
\aliascntresetthe{lem}

\newaliascnt{kor}{satz}
\newtheorem{kor}[kor]{Corollary}
\aliascntresetthe{kor}

\newaliascnt{prop}{satz}
\newtheorem{prop}[prop]{Proposition}
\aliascntresetthe{prop}

\newtheorem*{conj}{Vanishing Conjecture}

\theoremstyle{definition}
\newaliascnt{defn}{satz}
\newtheorem{defn}[defn]{Definition}
\newreptheorem{defn}{Definition}
\aliascntresetthe{defn}

\newaliascnt{bsp}{satz}
\newtheorem{bsp}[bsp]{Example}
\aliascntresetthe{bsp}

\newaliascnt{algo}{satz}
\newtheorem{algo}[algo]{Algorithm}
\aliascntresetthe{algo}

\theoremstyle{remark}
\newaliascnt{bem}{satz}
\newtheorem{bem}[bem]{Remark}
\aliascntresetthe{bem}

\crefformat{equation}{#2(#1)#3}
\crefformat{satz}{#2theorem~#1#3} 
\crefformat{defn}{#2definition~#1#3} 
\crefformat{lem}{#2lemma~#1#3} 
\crefformat{kor}{#2corollary~#1#3} 
\crefformat{prop}{#2proposition~#1#3} 
\crefformat{chapter}{#2chapter~#1#3}
\crefformat{bem}{#2remark~#1#3}
\crefformat{thm}{#2theorem~#1#3}
\crefformat{bsp}{#2example~#1#3}
\crefformat{appendix}{#2appendix~#1#3}
\crefformat{algo}{#2algorithm~#1#3}

\numberwithin{equation}{section}

\newcommand{\card}{\#} 
\DeclareMathOperator{\spec}{Spec}
\DeclareMathOperator{\Spec}{Spec}
\DeclareMathOperator{\quot}{Quot}
\DeclareMathOperator{\Quot}{Quot}

\DeclareMathOperator{\Proj}{Proj}

\DeclareMathOperator{\Bl}{Bl}
\DeclareMathOperator{\Div}{div}
\DeclareMathOperator{\CH}{CH}
\DeclareMathOperator{\RK}{\mathscr{R}}

\DeclareMathOperator{\Tor}{Tor}
\DeclareMathOperator{\len}{length}
\DeclareMathOperator{\Hom}{Hom}

\DeclareMathOperator{\mycolim}{colim}
\newcommand{\colim}{\mathop{\mycolim}}

\DeclareMathOperator{\trdeg}{trdeg}
\DeclareMathOperator{\ldeg}{ldeg}

\DeclareMathOperator{\im}{Im}
\DeclareMathOperator{\Rat}{Rat}

\newcommand{\IId}{\mathcal{I}}

\newcommand{\mId}{\mathfrak{m}}

\newcommand{\Oo}{\mathcal{O}}
\newcommand{\Ok}{\mathcal{K}}

\newcommand{\IZ}{\mathbb{Z}}
\newcommand{\IQ}{\mathbb{Q}}

\newcommand{\IN}{\mathbb{N}}

\newcommand{\IF}{\mathbb{F}}

\newcommand{\angles}[1]{\langle#1\rangle}

\newcommand{\blc}{\mathrm{B}}  
\newcommand{\pr}{\mathrm{pr}}    

\newcommand{\KC}{\mathcal{C}}

\newcommand{\rat}{\mathrm{Rat}}

\newcommand{\cadiv}{\mathrm{CaDiv}}
\newcommand{\sset}{\mathrm{sSet}}
\newcommand{\set}{\mathrm{Set}}

\newcommand{\Part}{\mathcal{P}}

\newcommand{\poset}{\mathrm{Poset}}

\newcommand{\sd}{\mathrm{sd}}
\newcommand{\unt}{\mathrm{sd}}

\newcommand{\BIGOP}[1]{\mathop{\mathchoice%
{\raise-0.22em\hbox{\huge $#1$}}%
{\raise-0.05em\hbox{\Large $#1$}}{\hbox{\large $#1$}}{#1}}}

\newcommand{\BIGboxplus}{\mathop{\mathchoice%
{\raise-0.35em\hbox{\huge $\boxplus$}}%
{\raise-0.15em\hbox{\Large $\boxplus$}}{\hbox{\large $\boxplus$}}{\boxplus}}}

\begin{document}
    \newlength{\drop}

    \title{A Simplicial Calculus for Local Intersection Numbers at Nonarchimedian Places on Products of Semi-stable Curves}
    \date{\today}
    \author{Johannes Kolb}

    \maketitle

    \begin{abstract}
        We analyse the subring of the Chow ring with support generated by the
        irreducible components of the special fibre of the Gross-Schoen
        desingularization of a $d$-fold self product of a semi-stable curve over the spectrum of a
        discrete valuation ring.
        For this purpose we develop a calculus 
        which allows to determine intersection numbers in the special fibre explicitly. 
        As input our simplicial calculus needs only combinatorial data of the special fibre. 
        It yields a practical procedure for calculating even self intersections in the special fibre.
        The first ingredient of our simplicial calculus is a localization formula,
        which reduces the problem of calculating intersection numbers to a special situation.
        In order to illustrate how our simplicial calculus works, we calculate all
        intersection numbers between divisors with support in the special fibre 
        in dimension three and four. 
        The localization formula and the general idea were already presented
        for $d=2$ in a paper of Zhang \cite[Ch. 3]{zhang}. 
        In our present work
        we achieve a generalisation to arbitrary $d$.
    \end{abstract}

    \tableofcontents

    \section{Introduction} 

\begin{par}
    Let $R$ be a complete discrete valuation ring with algebraically closed residue class field $k$. 
    We denote the quotient field $\quot(R)$ by $K$ and a uniformizing element with $\pi \in R$.
    Furthermore let $S$ denote the scheme $\spec{R}$ with generic point $\eta$ and special point $s$.
    Let $X$ be a regular strict semi-stable $S$\nobreakdash-scheme.
    We denote by $\cadiv_{X_s}(W)$
    the group of Cartier divisors on $X$ with support in the special fibre $X_s$.
    Intersection theory with support yields a product
    \[
        \Big(\cadiv_{X_s}(X)\Big)^p \to \CH^{p-1}(X_s),
    \]
    where $\CH^p$ denotes the Chow group in codimension $p$.
    If $X/S$ is proper of dimension $\dim(X)=d+1$,
    then $X_s$ is proper over a field and therefore 
    there is a degree map $\ldeg: \CH^d(X_s) \to \IZ$.
    We are interested in the pairing
    \begin{equation}
        \label{einl-paar}
        \begin{aligned}
            \big(\cadiv_{X_s}(X)\big)^{d+1} &\to \IZ, \\
            (C_0, \ldots, C_d) &\mapsto \ldeg(C_0 \cdot \cdots \cdot C_d)
        \end{aligned}
    \end{equation}
    given by the intersection pairing and the degree map.
\end{par}
\begin{par}
    Let us first look at a simple example: Assume that $X$ is a regular strict
    semi-stable model of a smooth proper curve $X_\eta$ over $K$.
    By the semi-stable reduction theorem, each smooth curve over $K$ 
    has a regular strict semi-stable model after finite base-change.
    We denote by $\Gamma(X)$ the dual graph of $X_s$. 
    Then the pairing \cref{einl-paar} can be calculated by counting suitable edges in $\Gamma(X)$:
    For instance, the self-intersection number of an irreducible 
    component $C \subseteq X_s$ is given by the number of edges in $\Gamma(X)$
    connected to $C$.
\end{par}
\begin{par}
    We deduce an analog to this description in the following higher-dimensional setting,
    in which an explicit construction of a semi-stable model is still possible:
    Let $X$ be a regular strict semi-stable model of a smooth curve over $K$. 
    Then a model of $(X_\eta)^d$ is given by the $d$-fold product 
    $X \times_R \cdots \times _R X$. 
    As already shown by Gross and Schoen (\cite{gross}),
    this model can be desingularized to a regular strict semi-stable scheme.
    Using as additional data an ordering on the set $X_s^{(0)}$ 
    of irreducible components of $X_s$ we can make this desingularization
    canonical, therefore we get a well-defined desingularization $W$ of the scheme $X^d$. 
\end{par}
\begin{par}
    As replacement for the reduction graph in higher dimension we use the incidence relations in $W_s$ to define
    a simplicial set $\RK(W)$, the simplicial reduction set. The underlying set of its geometric realization 
    is just the product $\RK(W)=\Gamma(X)^d$
    endowed with a triangulation depending on the chosen order on $X_s^{(0)}$.
    We introduce a calculus on the subring of the Chow group with support in the
    special fibre, which is generated by the classes of the irreducible
    components of $W_s$, which enables us to calculate intersection numbers of components of the special fibre.
    This simplicial calculus uses only explicitly given rational equivalences which arise 
    from the combinatoric of $\RK(W)$.
    It is therefore described using a ring $\KC(\RK(W))$, which is generated by the $0$-simplices in $\RK(W)$ 
    (see \cref{schnitt-chw}).
\end{par}
\begin{par}
    The calculus determines intersection numbers in a very controlled way, which ensures that the numbers are localized
    in a certain way. 
    It is enough to determine these numbers in the following local situation:
    Let $\bar L$ denote the regular strict semi-stable model $\bar L:=\Proj(x_0x_1 - \pi z^2)$.
    The special fibre $\bar L_s$ consists of two irreducible components $\Div(x_0)$ and $\Div(x_1)$, 
    which we order by $\Div(x_0) < \Div(x_1)$.
    Let $d \in \IN$ and $M$ be the Gross-Schoen desingularization described above of the product $L^d$. 
    Then $\RK(M)$ is a cube with its the standard desingularization.
\end{par}
\begin{par}
    The reduction simplicial set $\RK(W)$ can be split into such cubes:
    Denote by $\Gamma(X)^1$ the edges of the graph $\Gamma(X)$.
    For each tuple $\gamma=(\gamma_1, \ldots, \gamma_d)$ with 
    $\gamma_1, \ldots, \gamma_d \in \Gamma(X)^1$ there exists by functoriality
    an associated embedding
    $i_\gamma:\RK(M) \to \RK(W)$, 
    which induces a morphism
    $i_\gamma^*: \KC(\RK(W)) \to \KC(\RK(M))$.
    These morphisms localise the problem of calculating intersection numbers:
\end{par}
\begin{satz}
    Let $\alpha \in \KC(W)^{d+1}$. Then the equation
    \[
        \ldeg_W(\alpha) = \sum_{\gamma=(\gamma_1, \ldots, \gamma_d) 
            \in (\Gamma(X)^1)^d}
        \ldeg_M(i_\gamma^*(\alpha))
    \]
    holds.
\end{satz}
\begin{par}
    This theorem justifies a closer look on the local situation $M$. 
    We may subscript the vertices of $\RK(M)$ with coordinate vectors $v \in \IF_2^d$.
    Thus $\{C_v \mid v \in \IF_2^d\}$ is a basis of $\KC(\RK(M))^1$. The discrete Fourier transforms
    \[
        F_v := \sum_{w \in \IF_2^d} (-1)^{\angles{v,w}} C_w
    \]
    yield another basis of $\KC(\RK(M))^1_\IQ:=\KC(\RK(M))^1 \otimes_\IZ \IQ$. 
    It turns out that in this basis the intersection numbers
    are relatively easy to describe for $d \in \{2,3\}$:
\end{par}
\begin{repsatz}{schnitt-zahlen-d-2}
    Let $d=2$ and $v_1, v_2, v_3$ vectors in $\IF_2^2$. 
    Then the following holds:
    \[
        \ldeg(F_{v_1}F_{v_2} F_{v_3}) = \begin{cases}
            -32 & \textrm{ if } v_1=v_2=v_3 = (1,1), \\
            16  & \textrm{ if } \{v_1,v_2,v_3\} = \{(1,0),(0,1),(1,1)\}, \\
            0   & \textrm{ otherwise}.
        \end{cases}
    \]
\end{repsatz}
\begin{par}
    Using a different desingularization procedure, Zhang already studied a similar situation 
    in \cite[3.1]{zhang}.
\end{par}
\begin{par}
    Note that for the case $d=3$ the symmetric group $S_3$ acts on vectors in $\IF_2^3$ by 
    permuting the standard basis and the symmetric group $S_4$ acts on tuples
    $(v_0,\ldots,v_3)$ by a permutation of the elements. For $\sigma \in S_4, \tau \in S_3$ denote
    the successive appliction of both operations by
    \[
        (v_0, \ldots v_3)^{\sigma,\tau} := (v^\tau_{\sigma(0)}, \ldots, v^\tau_{\sigma(3)}).
    \]
\end{par}
\begin{repsatz}{schnitt-zahlen-d-3}
    Let $V=(v_0, \ldots v_3) \in (\IF_2^3)^4$ be a 4-tuple of vectors in $\IF_2^3$.
    Then the intersection numbers in $\KC(I^3)_\IQ$ are
    \[
        \ldeg(F_{v_0}F_{v_1}F_{v_2}F_{v_3}) = \begin{cases}
            -64   & \textrm{ if } V=(100,010,001,111)^{\sigma,\tau}\textrm{ for }\sigma \in S_4, \tau \in S_3, \\
            -64   & \textrm{ if } V=(100,010,101,011)^{\sigma,\tau}\textrm{ for }\sigma \in S_4, \tau \in S_3, \\
            -64   & \textrm{ if } V=(100,110,101,111)^{\sigma,\tau}\textrm{ for }\sigma \in S_4, \tau \in S_3, \\
            128   & \textrm{ if } V=(100,011,011,111)^{\sigma,\tau}\textrm{ for }\sigma \in S_4, \tau \in S_3, \\
            128   & \textrm{ if } V=(100,111,111,111)^{\sigma,\tau}\textrm{ for }\sigma \in S_4, \tau \in S_3, \\
            128   & \textrm{ if } V=(110,110,101,011)^{\sigma,\tau}\textrm{ for }\sigma \in S_4, \tau \in S_3, \\
            -128  & \textrm{ if } V=(110,101,111,111)^{\sigma,\tau}\textrm{ for }\sigma \in S_4, \tau \in S_3, \\
            512   & \textrm{ if } V=(111,111,111,111)^{\sigma,\tau}\textrm{ for }\sigma \in S_4, \tau \in S_3, \\
            0    & \textrm{ otherwise }.
        \end{cases}
    \]
\end{repsatz}
\begin{par}
    Especially interesting is the fact that many of the intersection numbers vanish. 
    We may therefore propose a vanishing conjecture, which we will describe in the following:
\end{par}
\begin{defn}
    Let $\Part=\{P_1, \ldots, P_l\}$ be a partition of the set $\{1, \ldots, d\}$
    and $v=(v_1, \ldots, v_d) \in \IF_2^d$.
    Then set 
    \[
        \alpha(\Part, v) := \# \{ i \in \{1, \ldots, l\} \mid \exists j \in P_i, v_j = 1 \}.
    \]
\end{defn}
\begin{repdefn}{limit-konv-bed}
    Let $d \in \IN$. We say that $d$ verifies the vanishing condition, 
    iff for each $\Part$ a partition of $\{1, \ldots, d\}$
    and $v_0, \ldots, v_d \in \IF_2^d$ with
    \[
        \sum_{i} \alpha(\Part, v_i) < d+|\Part|
    \]
    the intersection number
    \[
        \ldeg( \prod_{i} F_{v_i} )
    \]
    vanishes.
\end{repdefn}
\begin{conj}
    The vanishing condition holds for arbitrary $d \in \IN$.
\end{conj}
\begin{par}
    By \cref{schnitt-zahlen-d-2} and \cref{schnitt-zahlen-d-3} this conjecture is true for $d=2$ and $d=3$. 
    Using the computer algrebra system Sage \cite{sage} we were also able to verify the cases $d=4$ and $d=5$.
    This computer verification is a first demonstration of the power of the simplicial calculus developed in this paper.
\end{par}
\begin{par}
    Each verification of the vanishing conjecture for a particular d has
    non-trivial consequences which were our original motivation to consider
    this conjecture.  Namely we can derive a formula for the arithmetic
    intersection numbers of suitable adelic metrized line bundles on the d-fold
    self-product of any curve  in purely combinatorial and elementary analytic
    terms on the associated reduction complex. Details can be found in \cite{publ2}.
\end{par}



\section{Regular Strict Semi-stable Schemes} 
\label{rss}

\begin{par}
    For regular strict semi-stable curves the reduction graph is a good combinatorial
    invariant. We generalise this concept for arbitrary regular strict semi-stable schemes
    and define the simplicial reduction set.
    This combinatorial object is used later to compute intersection numbers
    in the special fibre.
\end{par}

\begin{par}
    Let $R$ denote a complete discrete valuation ring
    with algebraically closed residue field $k$. 
    The scheme $S:=\Spec{R}$ consists of the generic point $\eta$ and
    the special point $s$. Furthermore let $\pi$ denote an uniformizer of $R$.
\end{par}
\begin{par}
    Let $X$ be a scheme. The set of points of codimension $q$
    is denoted by $X^{(q)}$.
    In particular, we get the irreducible components by
    $X^{(0)}:=\{p \in X \mid \dim \Oo_{X,p} = 0\}$.
\end{par}
\begin{defn}
    \label{rss-def}
    \begin{par}
        \cite[2.16]{deJong}
        Let $S:=\Spec{R}$ be the spectrum of a discrete valuation ring $R$
        with algebraic closed residue field $k$.
        Let $X$ be an integral flat and separated $S$\nobreakdash-scheme of finite type.
        We call $X$ \emph{regular strict semi-stable}, if the following
        properties hold:
        \begin{enumerate}[(i)]
        \item
            The generic fibre $X_\eta$ is smooth,
        \item
            the special fibre $X_s$ is reduced,
        \item  
            each irreducible component $C$ of $X_s$ is a Cartier divisor on $X$,
            and
        \item
            if $C_1, \ldots C_m$ is a subset
            of irreducible components of $X_s$,
            then the scheme-theoretic intersection
            $C_1 \cap \ldots \cap C_m$ is either smooth
            or empty.
        \end{enumerate}
    \end{par}
    \begin{par}
        A regular strict semi-stable $S$\nobreakdash-variety of dimension $2$
        has relative dimension 1 over $S$ and is therefore called
        \emph{regular strict semi-stable curve over $S$}.
    \end{par}
\end{defn}


\begin{bsp}
    \label{rss-std-bsp}
    The affine scheme
    $\Spec{R[x_0, \ldots, x_n]/(x_0 \cdot\cdots\cdot x_n - \pi)}$ 
    is a regular strict semi-stable $S$\nobreakdash-variety.
\end{bsp}
\begin{par}
    According to Urs Hartl \cite{hartl} every regular strict semi-stable
    $S$\nobreakdash-variety is of this type:
\end{par}

\begin{satz}
    \label{ss-chara}
    A $S$\nobreakdash-variety $X$ is regular strict semi-stable 
    iff the generic fibre $X_\eta$ is smooth and for each closed point $x \in X_s$ 
    exists 
    an open neighbourhood $U$ in $X$, a number $m \in \IN$ and a smooth morphism
    \[ f: U \to L_m:=\Spec{R[x_0, \ldots x_{m}]/(x_0 \cdot\cdots\cdot x_{m} - \pi)}. \]
    The morphism can be chosen such
    that it maps the point $x$ to the ``origin'', i.e., the point $p_0$ given by the
    ideal $(x_0, \ldots x_m)$.
\end{satz}
\begin{proof}
    \cite[Prop 1.3]{hartl}
\end{proof}
\begin{par}
    We use this fact in form of the following easy corollaries:
\end{par}
\begin{kor}
    \label{ss-inv-glatt}
    Let $X,Y$ be integral, flat, separated $S$-schemes of finite type and
    $f: X \to Y$ a smooth morphism.
    If $Y$ is regular strict semi-stable, then so is $X$.
\end{kor}

\begin{kor}
    \label{ss-chara-l}
    Let $X$ be a regular strict semi-stable $S$\nobreakdash-curve. 
    Then each closed point $x \in X_s$ is either smooth
    or has an open neighbourhood $U$ and an \'etale map
    \[ U \to \Spec{R[x_0,x_1]/(x_0x_1 - \pi)}. \]
\end{kor}
\begin{proof}
    According to \cref{ss-chara} there is a smooth morphism
    \[ f: U \to \Spec{R[x_0, \ldots, x_{m}]/(x_0 \cdot\cdots\cdot x_{m} - \pi)} \]
    of an open neighbourhood $U$ of $x$. By dimension theory only $m=0$ and $m=1$ is possible.
    If $m=0$, then $U \to S$ is smooth. If otherwise $m=1$, the smooth morphism $f$
    has relative dimension $0$ and is therefore \'etale.
\end{proof}
\begin{par}
    It follows from \cref{ss-chara-l} that for \'etale local questions
    we can restrict ourselves to the model scheme
    $L=L_1:=\Spec{R[x_0,x_1]/(x_0 x_1 - \pi)}$. 
    The special fibre of this scheme consists of two components, which have a proper
    intersection in one point, the ``origin'' given by the ideal $(x_0,x_1)$.
\end{par}

\begin{par}
    Let $X$ be a regular strict semi-stable scheme. 
    A set of pairwise different components
    of the special fibre $X_s$ intersects properly by \cref{rss-def}(iv) 
    and we will show in \cref{chow-semi-eigtl}
    that this is an intersection of multiplicity $1$
    (in the sense of intersection theory).
    We may thus expect that a part of the Chow group is determined 
    only by the incidence relations between the components.
    To get better functorial properties we endow these incidence relations 
    with the structure of a simplicial set, the simplicial reduction set.
\end{par}
\begin{par}
    For this definition it is necessary to choose 
    a total ordering on $X_s^{(0)}$, 
    the components of the special fibre. 
    This is a transitive, antisymmetric and reflexive relation $\leq$,
    by which each two elements $C_1,C_2 \in X_s^{(0)}$ are comparable.
    Furthermore we employ the usual definitions from the theory of simplicial sets: 
    By $\Delta$ we denote the simplicial category, this is the category consisting
    of the ordered sets $[n]:=\{0, \ldots n\}$ for each $n \in \IN_0$
    as objects
    and monotonically increasing maps as morphisms.
    Some basic facts about partial orders and the simplicial category 
    are outlined in \cref{sk-kap}.
\end{par}

\begin{defn}
    \label{rss-def-beta}
    Let $X$ be a regular strict semi-stable $S$-scheme and $\leq$ a total ordering
    on $X_s^{(0)}$. 
    For each morphism of ordered sets $\beta: [n] \to X_s^{(0)}$, 
    i.e., a monotonically increasing map, we denote the scheme-theoretic intersection
    \[ 
        [\beta] := \beta(0) \cap \cdots \cap \beta(n)
    \]
    by $[\beta]$.
\end{defn}
\begin{bem}
    Let $X,\leq$ be as above, $\beta: [m] \to X_s^{(0)}$ a morphism of ordered sets
    and $f: [n] \to [m]$ a morphism of the simplicial category $\Delta$.
    Then
    \[ [\beta] \subseteq [\beta \circ f] \]
    holds.
    Since $[\beta]$ is smooth over $k$ (\cref{rss-def} (iv)), 
    the irreducible components $[\beta]^{(0)}$ 
    are actually connected components (\cite[Cor 4.2.17]{liu})
    and therefore there is a canonical morphism
    \begin{equation}
        \label{rk-ss-abb}
        f_\beta: [\beta]^{(0)} \to [\beta \circ f]^{(0)}
    \end{equation}
    which maps each point from $[\beta]^{(0)}$ onto its containing
    connected component from $[\beta \circ f]$.
\end{bem}
\begin{defn}
    \label{rss-def-rk}
    \begin{par}
        Let $X$ be a regular strict semi-stable scheme on $S$ and $\leq$ a total ordering
        on $X_s^{(0)}$.
        The \emph{simplicial reduction set} of $X$ is the simplicial set
        $\RK(X): \Delta \to \set$ defined on objects $[n] \in \Delta$ by
        \[ \RK(X)_n := \RK(X)([n]) := 
            \coprod_{\beta \in \hom([n], X_s^{(0)})}
            [\beta]^{(0)}
        \]
        and on morphisms $f: [n] \to [m]$ by
        \[ 
            \RK(X)(f) = 
            \left[
                \coprod_{\beta \in \hom([m], X_s^{(0)})}
                f_\beta
            \right]:
            \RK(X)_m \to \RK(X)_n.
        \]
        In the last equation $f_\beta$ is the map from \cref{rk-ss-abb}.
    \end{par}
    \begin{par}
        If $\dim(X)=2$, i.e., $X$ is a $S$\nobreakdash-curve,
        we call $\RK(X)$ also \emph{reduction graph} and denote it by $\Gamma(X)$.
    \end{par}
\end{defn}


\begin{par}
    A simplicial set is determined easily if each simplex is uniquely given 
    by its vertices. We call these \emph{simplicial sets without multiple simplices}
    (compare \cref{sk-einfach-kompl}).
    We may tests this property using the following criterium:
\end{par}
\begin{prop}
    \label{ss-redkomp-einfach}
    \begin{par}
        Let $X$ be a regular strict semi-stable $S$\nobreakdash-scheme with a total ordering $\leq$
        on $X_s^{(0)}$. The simplicial set $\RK(X)$ is a simplicial set without multiple
        simplices iff for each set $\{C_1, \ldots C_k\}$ of components of $X_s$
        the intersection $C_1 \cap \cdots \cap C_k$ is connected.
    \end{par}
    \begin{par}
        In this case there is a canonical bijection
        \begin{equation}
            \label{ss-redkomp-einfach-iso}
            \RK(X)_k
            \simeq \{ C_0 \leq \cdots \leq C_k \mid C_0, \ldots, C_k \in \RK(X)_0, 
                C_0 \cap \cdots \cap C_d \neq \emptyset\}
        \end{equation}
        between the $k$\nobreakdash-simplices and ascending chains of components in $(X_s)^{(0)}$
        with non-empty intersection.
    \end{par}
\end{prop}
\begin{proof}
    Let $\sigma \in \RK(X)_n$ be an $n$-simplex of the reduction complex.
    It is given by a pair 
    $(\beta,p)$ with $\beta: [n] \to X_s^{(0)}$ and $p \in [\beta]^{(0)}$.
    An easy computation shows that the vertices of $\sigma$ are given by
    $\beta(0), \ldots \beta(n)$.
    The simplex $\sigma$ is therefore uniquely determined by its vertices
    iff $[\beta]^{(0)}$ is a singleton, which means $[\beta]$ is connected.
    The bijection in \cref{ss-redkomp-einfach-iso} is then given by
    \[
        \sigma = (\beta,p) \mapsto (\beta(0) \leq \cdots \leq \beta(n)).
    \]
\end{proof}

\begin{bem}
    From now on we restrict ourselves to regular strict semi-stable schemes
    having a reduction set without multiple simplices,
    since \cref{ss-redkomp-einfach} gives a comfortable description 
    of the reduction set.
    At least for curves this restriction is not essential:
    By a suitable base change $S_n \to S$ and a subsequent desingularization
    each regular strict semi-stable $S$\nobreakdash-curve can be transformed
    into a regular strict-semi-stable $S_n$\nobreakdash-curve 
    without multiple simplices.
    This process is recalled in \cref{desi}.
\end{bem}
\begin{bsp}
    The affine $S$\nobreakdash-scheme $L_m:=\Spec{R[x_0, \ldots x_m]/(x_0 \cdot\cdots\cdot x_n - \pi)}$
    from \cref{rss-std-bsp} is regular strict semi-stable having a simplicial reduction set
    without multiple simplices.
    The components of the special fibre of $L_m$ are given by the
    ideals $(x_i)$ (i=0, \ldots, m).
    We endow them with the order $(x_i) \leq (x_j)$ where $(i \leq j)$.
    Since every intersection of such components is connected and non-empty,
    the simplicial reduction set $\RK(L_m)$ is free of multiple simplices and 
    $\RK(L_m)_k = \Hom_{\Delta}([k], [m])$. 
    Therefore $\RK(L_m)$ is the standard-$m$-simplex $\Delta[m]$.
\end{bsp}
\begin{bsp}
    Let $X$ be a regular strict semi-stable $S$\nobreakdash-curve. 
    The simplicial reduction set $\Gamma(X)$ has dimension 1 and is therefore 
    an ordered graph.
    It is easy to see that $\Gamma(X)$ coincides with the usual definition
    of the reduction graph \cite[9.2]{neron}.
\end{bsp}

\begin{par}
    The simplicial reduction set is functorial for generic flat morphisms:
\end{par}
\begin{prop}
    \label{redkomp-pushforward}
    Let $X$ and $Y$ be regular strict semi-stable $S$\nobreakdash-schemes with
    total orderings $\leq_X$ resp. $\leq_Y$ on $X_s^{(0)}$ resp. $Y_s^{(0)}$.
    Let $f:X \to Y$ be a morphism which is flat in a open subset 
    which contains all generic points of $Y_s$.
    \begin{enumerate}[(i)]
        \item
            Then there is a morphism
            \begin{equation}
                \label{redkomp-pushforward-komp1}
                \begin{split}
                    f_*: X_s^{(0)} &\to Y_s^{(0)}, \\
                    C &\mapsto \overline{f(C)}.
                \end{split}
            \end{equation}
        \item
            If $\RK(X)$ and $\RK(Y)$ are without multiple simplices 
            and $f_*$ from
            \cref{redkomp-pushforward-komp1} preserves the order,
            we may extend $f_*$ to a morphism of simplicial sets by
            \begin{equation}
                \label{redkomp-pushforward-redkomp}
                \begin{split}
                    f_*: \RK(X) &\to \RK(Y),\\
                    (C_0 \leq_X \cdots \leq_X C_k) &\mapsto 
                    (\overline{f(C_0)} \leq_Y \cdots \leq_Y \overline{f(C_k)}).
                \end{split}
            \end{equation}
            In this definition the simplices of $\RK(X)$ and $\RK(Y)$ are
            given by the bijection \cref{ss-redkomp-einfach-iso} in
            \cref{ss-redkomp-einfach}.
    \end{enumerate}
\end{prop}

\begin{proof}
    \begin{par}
        Let $C \in X_s^{(0)}$ be a component of $X_s$. For the first claim 
        it suffices to show $\overline{f(C)} \in Y_s^{(0)}$.
        Let $p \in X_s$ be the generic point of $C$. After restricting to a neighbourhood
        of $p$ we may assume that $f$ is flat.
        Then also the base change $f': X_s \to Y_s$ is flat and by 
        \cite[IV, \S 2, Cor (2.3.5) (ii)]{ega42} the closure $\overline{\{f(p)\}}$
        is an irreducible component of $Y_s$.
    \end{par}
    \begin{par}
        For the second claim we have to show that
        \cref{redkomp-pushforward-redkomp} is well-defined. 
        Let $C_0 \leq_X \cdots \leq_X C_k$ be an ascending chain representing a
        $k$-simplex of $\RK(X)$, i.e., with $C_0 \cap \cdots \cap C_k \neq \emptyset$. 
        By assumption, 
        $\overline{f(C_0)} \leq_Y \cdots \leq_Y \overline{f(C_k)}$ also holds.
        Eventually each point of $C_0 \cap \cdots \cap C_k$ is mapped by $f$
        onto a point of $\overline{f(C_0)} \cap \cdots \cap \overline{f(C_k)}$,
        therefore the chain $(\overline{f(C_0)} \leq_Y \cdots \leq_Y \overline{f(C_k)})$
        represents an $k$-simplex of $\RK(Y)$.
    \end{par}
\end{proof}
\begin{par}
    We may determine the simplicial reduction set locally:
\end{par}
\begin{prop}
    \label{redkomp-lok-ber}
    Let $X$ be a regular strict semi-stable $S$\nobreakdash-scheme with total ordering $\leq$ 
    on $X_s^{(0)}$.
    \begin{enumerate}[(i)]
        \item
            If $U \subseteq X$ is an open subset,
            then $\leq_X$ induces a total ordering $\leq_U$ on $U_s^{(0)}$
            and there is a canonical monomorphism $\RK(U) \to \RK(X)$.
        \item
            Let $\mathcal{U}=\big(U_i\big)_{i \in I}$ be a covering system of open sets for $X$.
            If for each two sets $U_i, U_j \in \mathcal{U}$ the intersection 
            $U_i \cap U_j$ has a covering with sets from $\mathcal{U}$,
            then
            \[ \RK(X) = \colim\limits_{i \in I} \RK(U_i). \]
    \end{enumerate}
\end{prop}
\begin{proof}
    \begin{par}
        Claim (i) is an immediate consequence of the definition of $\RK(X)$. 
        We now show (ii).
        The universal property of the colimit yields
        a unique morphism 
        \begin{equation}
            \label{redkomp-lok-ber-mor}
            \varphi: \colim\limits_{i \in I} \RK(U_i) \to \RK(X),
        \end{equation}
        induced by the inclusions $U_i \to X$. 
        We have to show that this is an isomorphism.
    \end{par}
    \begin{par}
        Let $k \in \IN$ and $\sigma \in \RK(X)_k$ be an $k$-simplex. 
        By definition $\sigma$ is given by $\beta \in \Hom([n], X_s^{(0)})$ and 
        $p \in [\beta]^{(0)}$.
        Since $\mathcal{U}$ is a covering system for $X$, there is an $U \in \mathcal{U}$
        with $p \in U$ and thus $\sigma \in \RK(U)$. Morphism \cref{redkomp-lok-ber-mor}
        is therefore surjective.
    \end{par}
    \begin{par}
        To proof injectivity let $k \in \IN$, $U, U' \in \mathcal{U}$ and
        $\sigma \in \RK(U)_k$, $\sigma' \in \RK(U')_k$ be such that
        $\varphi(\sigma)=\varphi(\sigma')$.
        As before let $\sigma$ be given by a pair $(\beta, p)$ with
        $\beta \in \Hom([n],U_s^{(0)}), p \in [\beta]^{(0)}$
        and $\sigma'$ be given by $(\beta', p')$.
        Then $\varphi(\sigma)=\varphi(\sigma')$ means $p=p'$ as points of $X$.
        Thus we may choose a neighbourhood $U'' \in \mathcal{U}$ of $p$ 
        with $U'' \subseteq U \cap U'$.
        The point $p$ also defines a simplex $\sigma'' \in \RK(U'')$.
        For the inclusions $i: U'' \to U$ and $i': U'' \to U'$ we have
        \[
            i_*(\sigma'') = \sigma, \quad i'_*(\sigma'') = \sigma' 
        \]
        and therefore $\sigma$ and $\sigma'$ agree in $\colim_{i \in I}\RK(U_i)$.
    \end{par}
\end{proof}

\begin{defn}
    \label{redkomp-std-umg}
    Let $X$ be a regular strict semi-stable scheme and $p \in X$ a closed point.
    Then an open subset $U \subseteq X$ is called
    \emph{standard neighbourhood of $p$},
    if either $U \subseteq X_\eta$ or 
    there is a number $m \in \IN$ and a smooth morphism
    \[
        f: U \to L_m:=\spec{R[x_0, \ldots, x_m]/(x_0 \cdot\cdots\cdot x_m - \pi)}
    \]
    such that $f_*: \RK(U) \to \RK(L_m)$ from \cref{redkomp-pushforward}(ii)
    is a bijection and $f(p) = p_0 := (x_0, \ldots, x_m, \pi)$ holds.
\end{defn}
\begin{prop}
    \label{redkomp-lok-bij} 
    Let $X$ be a regular strict semi-stable scheme, $U \subseteq X$ an open subset
    and $p \in X_s \cap U$ a closed point in the special fibre.
    Then there exists a standard neighbourhood of $p$ contained in $U$.
\end{prop}
\begin{proof}
    \begin{par}
        By \cref{ss-chara} there is an open subset $U'$ of $p$ and a smooth morphism
        \[ f: U' \to L_m \]
        with $f(p) = p_0$.
        We may shrink $U'$ such that each component of $U'_s$ contains $p$ 
        and $U' \subseteq U$ holds.
        By a further restriction we may assume that for each choice of irreducible
        components
        $C_0, \ldots, C_k \in (U')_s^{(0)}$ the intersection
        $C_0 \cap \cdots \cap C_k$ is connected.
        Thus the simplicial reduction set $\RK(U')$ is isomorphic 
        to the standard-$n$\nobreakdash-simplex 
        and has no multiple simplices.
        By \cref{sk-morph-mfs-eind} it suffices to show
        that $f$ induces a bijection $f_*: U_s^{(0)} \to (L_m)_s^{(0)}$
        on the $0$\nobreakdash-simplices:
    \end{par}
    \begin{par}
        Since $f$ is flat, the morphism $f_*$ exists by \cref{redkomp-pushforward}.
        To give the inverse map, let $C \in (L_m)_s^{(0)}$
        be an irreducible component of $(L_m)_s$. 
        By assumption $C$ is smooth and contains $p_0$.
        Therefore the scheme-theoretical preimage $f^{-1}(C)$ is
        a non-empty smooth $k$-scheme. This means, every irreducible component of 
        $f^{-1}(C)$ is a connected component of $f^{-1}(C)$. 
        Since $f$ is flat, each of these components is a component $U'_s$
        and therefore contains the point $p$.
        Thus $f^{-1}(C)$ consists of exactly one connected component of $U'_s$
	and the morphism $f_*$ is therefore bijective.
    \end{par}
\end{proof}
\begin{kor}
    \label{redkomp-std-ubd}
    Let $X$ be a regular strict semi-stable scheme. 
    Then every open subset $U \subseteq X$ can be covered
    by a system $(U_i)_{i \in I}$ of standard neighbourhoods.
    Therefore we may apply \cref{redkomp-lok-ber} (ii)
    on the covering
    \[ 
        \mathcal{U}:= \{ U \subseteq X \mid U \text{ is standard neighborhood
                of a closed point }
            p \in X\}
    \]
    and therefore
    \[ \RK(X) = \colim_{U \in \mathcal{U}} \RK(U) \]
    holds.
\end{kor}
\begin{proof}
    Let $U \subseteq X$ be an open subset. 
    We have to show that $U$ can be covered by standard neighbourhoods.
    Denote by $I$ the set of closed points of $U_s$. For each point $p \in I$ 
    we choose a standard neighbourhood $U_p \subseteq U$ by \cref{redkomp-lok-bij}.
    Then $\{ u_p \mid p \in I\} \subseteq \mathcal{U}$ 
    is an open covering of all closed points of $U_s$
    and by Hilbert's Nullstellensatz also a covering of $U_s$.
    Since $U_\eta \in \mathcal{U}$ we get a covering of $U$ by
    \[ 
        U = U_\eta \cup \left( \bigcup_{p \in I} U_p \right).
    \]
\end{proof}




\section{Desingularization} 
\label{desi}

\begin{par}
    Let $S:=\spec{R}$ be the spectrum of a complete discrete valuation ring
    with algebraically closed residue field $k$ and $X$ be
    a regular strict semi-stable $S$-scheme of dimension $2$,
    i.e., a $S$-curve.
    In this section we describe canonical desingularization methods
    for two important product situations:
\end{par}
\begin{par}
    For the first situation let $K_n/K$ be an algebraic field extension of $K=\quot{R}$
    of degree $n$ and $R_n$ the ring of integers in $K_n$.
    We will recall a desingularization for the product $X \times_S \spec{R_n}$.
    The second situation is the $k$-fold product $X^k$.
\end{par}
\begin{par}
    In both cases these desingularizations are well-known: 
    The base change $X \times_S \spec{R_n}$ is well-known for curves (\cite{delmum}).
    In \cite{gross}, Gross and Schoen investigate 
    products of regular strict semi-stable schemes
    in general. The same procedure is later described by Hartl \cite{hartl}
    together with a desingularization for $X \times_S \spec{R_n}$, where
    $X$ is any regular strict semi-stable scheme.
    In both techniques the order of the blow-ups used for desingularization
    is left open and therefore the result is not unique.
    A different order of blow-ups yields a different scheme and in general
    a different simplicial reduction set.
\end{par}
\begin{par}
    Since the order of blow-ups does not matter for curves,
    we can use for the situation $X \times_S \spec{R_n}$ with $X$ a $S$-curve
    the known method (\cite{delmum}, \cite{heinz}) 
    unmodified. For the product situation we use an adjusted form
    of the method described in \cite[Prop 6.11]{gross} and \cite[Prop 2.1]{hartl}.
\end{par}

\subsection{Ramified Base-Change of Curves} 
\label{desi-vbw-kap}

\begin{par}
    We begin with the case of ramified extensions. 
    Let $K_n/K$ be an algebraic field extension of $K=\quot{R}$
    of degree $n$ and $R_n$ the ring of integers in $K_n$. 
    Let $X$ be a regular strict semi-stable curve over $S=\spec{R}$. 
    The base change $X \times_{S} \spec{R_n}$ is not regular in general,
    but we can desingularize it in the same way as minimal models are constructed
    after base change (\cite[p. 3]{heinz}). 
    It turns out that this desingularization turns the reduction graph $\Gamma$ 
    into the $n$-fold subdivison $\unt_n(\Gamma)$.
\end{par}
\begin{satz}
   \label{desi-vbw}  
    Let $S:=\spec{R}$ be the spectrum of a complete discrete valuation ring
    and $S_n:=\spec{R_n}$ the spectrum of the ring $R_n$ above.
    Let $X$ be a regular strict semi-stable $S$-curve 
    with a total ordering on $X^{(0)}$, whose simplicial reduction set $\Gamma(X)$ has no
    multiple simplices.
    Let $X_n$ be the scheme obtained by
    blowing up $X \times_S S_n$ successively in all singular points, 
    blowing up the resulting scheme successively in all singular points, and so on $n/2$ times.
    Then $X_n$ is a regular strict semi-stable $S_n$ curve
    with
    $(X_n)_{\eta_n} = (X_\eta) \times_{\spec K} \spec K_n$.
    Furthermore there exists a total ordering of $(X_n)^{(0)}$ such
    that there exists a canonical isomorphism of simplicial reduction sets
    \[
        \Gamma(X_n) \simeq \unt_n(\Gamma(X)).
    \]
\end{satz}
\begin{proof}
    \begin{par}
        The desingularization process and the structure of the reduction graph of $X_n$
        is classical, the proof can be done by an explicit calculation like in \cite[p. 3]{heinz}. 
        For the statement about the simplicial reduction sets, we have to define a suitable
        total ordering on $(X_n)^{(0)}$, this can be done in the following way:
        Since $\Gamma(X_n)$ and $\sd_n(\Gamma(X))$ are isomorphic as unordered graphs,
        there is a bijection $\varphi: (X_n)_s^{(0)} \to \sd_n(\Gamma(X))$ compatible 
        with the graph structure. This means that two components $C,C' \in (X_n)_s^{(0)}$
        intersect iff there is an edge between the vertices $\varphi(C), \varphi(C') \in
        \sd_n(\Gamma(X))_0$.
    \end{par}
    \begin{par}
        By \cref{sk-unt-def} we may identify the elements of $\sd_n(\Gamma(X))$ with a subset of
        $(\Gamma(X)_0)^n = (X_s^{(0)})^n$. The lexicographical order on $(X_s^{(0)})^n$ 
        induces then a total ordering on $\sd_n(\Gamma(X))_0$ and by $\varphi$ also
        an total ordering on $(X_n)^{(0)}$. With this ordering, $\varphi$ is also 
        a morphism of simplicial sets.
        For details see \cite[Prop A.27, Lemma 2.6]{me}. 
    \end{par}
\end{proof}


\subsection{Desingularization of Products} 
\label{desi-aufprod}
\begin{par}
    Let $X$ be a regular strict semi-stable $S$-curve with total ordering on $X^{(0)}$.
    The product $X^d:=X \times_S \cdots \times_S X$ is in general
    not regular strict semi-stable. This can already be observed with 
    the standard scheme $X=L:=\spec{R[x_0,x_1]/(x_0x_1 - \pi)}$.
    We use a desingularization similar to \cite[Prop 6.11]{gross} and \cite{hartl}
    to get a regular strict semi-stable model $W(X,<,d)$ of $(X_\eta)^d$.
    The result of this process depends on a sequence of components, 
    which is chosen arbitrarily in the citations above.
    Using the total ordering on $X^{(0)}$ we are able to
    define a canonical sequence and thus get a result 
    with a well-defined simplicial reduction set.
\end{par}
\begin{par}
    The desingularization works as follows:
\end{par}
\begin{algo}
    \label{desi-prodkomp}
    Let $d \in \IN$ and $X$ be a regular strict semi-stable $S$-curve with total ordering $\leq$
    on $X_s^{(0)}$ and $\Gamma(X)$ a simplicial set without multiple simplices.
    We denote the product by $W_0 := X^d$. Since the components of $X_s$ are 
    geometrically integral, we can describe the irreducible components of $(W_0)_s$
    as product
    \[
        (W_0)_s^{(0)} = X_s^{(0)} \times \cdots \times X_s^{(0)}.
    \]
    We endow this product $(W_0)_s^{(0)}$ with the lexicographical order and denote the
    elements in ascending order $B_1, \ldots B_k$.
    Now denote by $B'_1$ the irreducible component $B_1$ endowed with the induced reduced
    structure and
    set $W_1:=\Bl_{B'_1}(W_0)$.
    Inductively let $B'_i \subseteq W_i$ be the strict transform of the
    irreducible component $B_i$ endowed with the induced reduced structure
    and set $W_{i+1}:=\Bl_{B'_i}(W_i)$.
    The last scheme in this chain, $W_{k}$, is also denoted by
    $W(X,\leq,d):=W_k$.
    These blowups introduce no new components in the special fibre $(W_k)_s$,
    so the lexicographical ordering on $(W_0)_s^{(0)}$ also induces 
    a total ordering on $(W_k)_s^{(0)}$.
\end{algo}
\begin{satz}
    \label{desi-prodkomp-ex}
    The scheme $W(X,\leq,d)$ constructed in \cref{desi-prodkomp} is regular strict semi-stable
    and the reduction set with respect to the lexicographical ordering induced by
    $(W_0)_s^{(0)}$ is given by
    $\RK(W) = \Gamma(X)^d$.
\end{satz}
\begin{par}
    Before we proof this fact, let us first assume that 
    $X$ is the standard scheme
    $X=L:=\spec{R[x_0,x_1]/(x_0x_1 - \pi)}$.
    In this case we may describe all schemes $W_i$ in \cref{desi-prodkomp}
    explicitly by an affine covering. In this covering we are able to check the 
    claims easily.
    The covering schemes are spectra of the following algebras:
\end{par}
\begin{defn}
    Let $l \in \IN_0$ be a natural number and $A$ a non-empty set. 
    We define a $R$-algebra $M(l,A)$ by
    \[
        M(l,A) := R[w_1, \ldots w_l][u_{a,0},u_{a,1} \mid a \in A]/ \IId_M,
    \]
    the quotient of a free commutative algebra generated by the elements
    $w_1, \ldots w_l$ and $\{u_{a,0},u_{a,1} \mid a \in A\}$
    by the Ideal $\IId_M$, which is generated by
    \[
        \IId_M := \left\{ 
            w_1 \cdot\cdots\cdot w_l u_{a0}u_{a1} - \pi, 
            u_{a0}u_{a1} - u_{a'0}u_{a'1} \mid a,a' \in A 
        \right\}.
    \]
\end{defn}
\begin{par}
    We first analyse the irreducible components and blow-ups of $\spec{M(l,A)}$.
\end{par}

\begin{lem}
    \label{de-lokblup-lem}
    Let $l \in \IN$ and $A$ be a finite set.
    \begin{enumerate}[(i)]
        \item
            The affine scheme $N:=\spec{M(l,A)}$ is integral. Its special fibre
            $N_s$ consists of the following components:
            For each $i \in \{1, \ldots l\}$ a Cartier divisor
            $D(i)$ given by the ideal $(w_i)$
            and
            for each mapping $t \in \Hom(A, \{0,1\})$ a component 
            $C(t)$
            given by the ideal 
            $(u_{k,t(a)} \mid a \in A)$.
        \item
            The blow-up $\tilde N := \Bl_{C(0)}(N)$ of $N$ at the component
            $C(0)$ corresponding to the zero mapping $0:A \to \{0,1\}$ 
            is covered by the family
            of affine schemes
            \[ \left( \tilde N_a:=\Spec(M(l+1, A\setminus\{a\}))\right)_{a \in A}.
            \]
        \item
            For each chart $\tilde N_a$ the irreducible components of 
            $(\tilde N_a)_s$ have the form $\tilde D_i$ with $i \in \{1, \ldots l+1\}$
            or $\tilde C(t)$ with $t \in \Hom(A \setminus\{a\}, \{0,1\})$ analogous to (i).
            Under the blow-up morphism $\psi_a: \tilde N_a \to N$ these components
            are mapped onto the following components of $N_s$:
            \begin{align*}
                \psi_a(\tilde D(i))   &= D(i) \quad \textrm{for $i \in \{1, \ldots, l\}$}, \\
                \psi_a(\tilde D(l+1)) &= C(0), \\
                \psi_a(\tilde C(t\mid_{A \setminus \{a\}}))  &= C(t) \quad \textrm{
                        for $t \in \Hom(A,\{0,1\})$ with $t(a)=1$}.
            \end{align*}
        \item
            Let $\psi_a$ be defined as above and denote by $E \in \cadiv(\tilde N)$ the
            exceptional divisor of the blow-up. 
            Then we have for any $\alpha \in A \setminus \{a\}$ the following 
            identities of principal Cartier divisors:
            \begin{align*}
                \psi_a^{-1}(\Div(u_{a,0}))E^{-1}\mid_{\tilde N_a} &= \Div(1), \\
                \psi_a^{-1}(\Div(u_{a,1}))      \mid_{\tilde N_a} &= \Div(\tilde
		u_{\alpha,0}\tilde u_{\alpha,1}), \\
                \psi_a^{-1}(\Div(u_{\alpha,0}))E^{-1}\mid_{\tilde N_a} &= \Div(\tilde
		u_{\alpha,0}), \\
                \psi_a^{-1}(\Div(u_{\alpha,1}))      \mid_{\tilde N_a} &= \Div(\tilde u_{\alpha,1}).
            \end{align*}
    \end{enumerate}
\end{lem}
\begin{proof}
    All these claims can be shown by explicite calculations just using 
    \cite[Lemma 8.1.4]{liu}. 
    For details see \cite[Lemma 2.14]{me}.
\end{proof}
\begin{par}
    Before we are able to return to the proof of \cref{desi-prodkomp-ex} with $X=L$ we
    need to introduce some additional notation. 
    On $W_0:=L^d$ we denote by $\pr_i: W_0 \to L$ the projection
    on the $i$-th factor.
    We denote the irreducible components of $L_s$ by $I_0, I_1$ with the
    obvious ordering $I_0 < I_1$. 
    Each component $C$ of $(W_0)_s$ is of the form
    $C = I_{t(1)} \times \cdots \times I_{t(d)}$
    for a mapping $t \in \Hom(\{1, \ldots d\}, \{0,1\})$. 
    We denote this mapping by $t_C$.
\end{par}
\begin{par}
    To apply \cref{desi-prodkomp} we arrange the components of $W_s$ lexicographically as
    $B_1, \ldots B_{2^d}$.
    In contrast to that we will also use the product order on $(W_s)^{(0)}$,
    which will be denoted by $\leq$:
    For two components $C,C' \in (W_s)^{(0)}$ set $C \leq C'$ iff
    for each $i$ the relation $t_C(i) \leq t_{C'}(i)$ holds.
    We denote by  $J:=\{C_1 < \cdots < C_{l+1}\} \subseteq \mathcal{P}(\RK(W_0)_0)$ the set of all strictly increasing chains with respect to $\leq$
    which cannot be refined. 
    For each $m \in \IN_0$ we define the subset $J^m$ by 
    \[ J^m := \left\{ (C_1 < \cdots < C_{l+1}) \in J \mid C_1 = \blc_1, 
    \{C_1, \ldots, C_{l} \} \subseteq \{\blc_1, \cdots \blc_m\} \right\}. \]
    Thus elements of $J^m$ are chains of components, which cannot be refined and 
    whose elements -- with the exception of the last element --
    are already blown-up in step $m$.
    In particular one gets $J^0 = \{(B_1)\}$.
\end{par}
\begin{lem}
    \label{desi-prodkomp-lok}
    Let $X=L$ and $W_m$ constructed as in \cref{desi-prodkomp}. 
    Then for each $m \in \{0, \ldots, 2^d\}$
    the following holds:
    \begin{enumerate}[(i)]
        \item
            The scheme $W_m$ is covered by the family of affine schemes
            \begin{equation}
                \label{desi-prodkomp-lok-ubd}
                \Big( \Spec{M(l, A(C_{l+1}))} \Big)_{ (C_1 < \cdots < C_{l+1}) \in J^m}
            \end{equation}
            where $A(C_l)$ denotes the subset
            $A(C):=t_C^{-1}(1) \subseteq \{1, \ldots d\}$.
            The irreducible components of $(W_m)_s$ are exactly the strict transforms
            of the components of $(W_0)_s$.
        \item
            In the chart $N(C_1, \ldots, C_{l+1}):=\Spec{M(l, A(C_{l+1}))}$ 
            corresponding to the chain 
            $(C_1 < \ldots < C_l < C_{l+1})$ the components $C_1, \ldots C_l$
            are given by the Cartier divisors $D(1), \ldots D(l)$.
            The chart also contains the strict transforms of all components
            $C \in (W_0)_s^{(0)}$ with $C \geq C_l$.
            Let $C \geq C_l$ be such a component, then 
            the strict transform of $C$ is given in the chart by
            $C(t_C \mid _{A(C_k)})$.
        \item
            Each centre of blow-up is the scheme-theoretic intersection
            of Cartier divisors.
            In particular the centre $B'_j$ is given by
            \[ 
                \blc'_j = \bigcap_{n=1}^d F_{j,n}
            \]
            with the Cartier divisors
            \[ 
                F_{j,n} := \pr_n^{-1}(\pr_n(\blc_j)) \sum_{i \leq m, \pr_n(\blc_i) =
                    \pr_n(\blc_j)} -\blc_i.
            \]
    \end{enumerate}
\end{lem}
\begin{proof}
    \begin{par}
        We show (i) and (ii) by induction on $m$.
        The claim is trivial for $m=0$:
        We have $J^0=\{B_1\}$ and $L^d$ is indeed covered by the affine scheme
        $\spec{M(0, \{1, \ldots d\})} \simeq L^d$. 
        The isomorphism can be chosen in such a way 
        that for each $n \in \{1, \ldots d\}$ 
        the projection $\pr_n$ induces a morphism on the global sections
        \[
            \Oo_L(L) \to \Oo_{N(\blc_1)}(N(\blc_1)),
        \]
        which maps $x_0$ onto $u_{n,0}$ and $x_1$ onto $u_{n,1}$.
        Then (ii) follows immediately.
    \end{par}
    \begin{par}
        Assume now that (i) and (ii) are true for $m \geq 0$. 
        We have to examine the blow-up $B_{m+1}=\Bl_{B'_{m+1}}(W_m)$ 
        and can restrict ourself to the charts of $W_m$ which contain 
        the component $B'_{m+1}$.
        This means we consider chains $(C_1 < \cdots < C_k) \in J^m$ with $C_k = B_{m+1}$ 
        with their associated chart $N=N(C_1, \ldots, C_k)$. 
        Since $t_{C_k} \mid _{A(C_k)} = 0$ we can apply \cref{de-lokblup-lem} to describe the blow-up:
    \end{par}
    \begin{par}
        The blow-up $\tilde N$ of $N$ in $C_k$ is covered by the charts
        \[
            \spec M(k+1, A(C_k)\setminus \{l\}), l \in A(C_k).
        \]
        If we associate to each $l \in A(C_k)$ the component $C'$ 
        with $A(C') = A(C) \setminus \{l\}$ and hence an element
        $(C_0 < \cdots < C_l < C') \in J^{m+1}$, 
        this proves (i).
        By \cref{de-lokblup-lem} again a component $C'' > C$ lies in the chart $N(C_0, \cdots < C_l < C')$
        iff $t_{C''}(a) = 1$ holds, which is equivalent to $C'' > C'$.
    \end{par}
    \begin{par}
        For the description of the centres of blow-up we show inductively using
        \cref{de-lokblup-lem}
        that in each chart $N(C_1, \ldots, C_{l+1}) \setminus W_m$ 
        the following identities hold:
        \[
            F_n \mid_{N(C_1, \ldots, C_{l+1})} = \begin{cases}
                V(u_{n,0})          & \text{if } \pr_n(C) = 0, \pr_n(C_l) = 0, \\
                V(1)                & \text{if } \pr_n(C) = 0, \pr_n(C_l) = 1, \\
                V(u_{n,1})          & \text{if } \pr_n(C) = 1, \pr_n(C_l) = 0, \\
                V(u_{a',0}u_{a',1}) \text{ for any } a' \in A(C_l) & \text{if }
                \pr_n(C) = 1, \pr_n(C_l) = 1. \\
            \end{cases}
        \]
        If the blow-up centre $B_j$ suffices $B_j > C_k$, we get
        \[
            \bigcap_{n=1}^d F_{j,n} \mid_{N(C_1, \ldots, C_{l+1})}
            = V\big((u_{a,t_{\blc_j}(a)} \mid a \in A(C_l)\big) = \blc_j.
        \]
        If otherwise $B_j \not > C_k$, we have
        \[
            \bigcap_{n=1}^d F_k \mid_{N(C_1, \ldots, C_{l+1})} = \emptyset.
        \]
        The last claim then follows directly.
    \end{par}
\end{proof}
\begin{proof}[Proof of \cref{desi-prodkomp-ex}]
    \begin{par}
        First assume $X=L$.
        By \cref{desi-prodkomp-lok} the scheme $W(L,<,d):=W_{2^d}$ is covered
        by affine charts of the form
        \[
            \Spec{M(k, \{a\})} \simeq \Spec{R[w_1, \ldots, w_l, w_{l+1}, w_{l+2}] / 
            (w_1 \cdot\cdots\cdot w_{l+2} - \pi)} 
        \]
        and is therefore regular strict semi-stable according to \cref{ss-chara}.
    \end{par}
    \begin{par}
        It remains to find an isomorphism of simplicial sets 
        $\RK(W_{2^d}) \simeq (\RK(L))^d$. 
        \cref{ss-redkomp-einfach} ensures that the reduction sets of $L$ and $W_{2^d}$ contain no multiple
        simplices,
        thus by \cref{redkomp-pushforward} the projection morphisms $\tilde \pr_i: M_m \to L$ 
        induce morphisms $(\tilde \pr_i)_*: \RK(W_{2^d}) \to \RK(L)$. 
        By the universal property of the product they induce a morphism
        $\RK(W_{2^d}) \to \RK(L)^d$ and we have to show injectivity and surjectivity.
        On the level of $0$-simplices this is trivial, 
        since the blow-up morphism $W_{2^d} \to W_0 = L^d$ do not induce
        additional components in the special fibre, thus
        $(W_{2^d})_s^{(0)} \simeq (W_0)_s^{(0)} \simeq (L_s^{(0)})^d$.
    \end{par}
    \begin{par}
        To show injectivity and surjectivity for $k$-simplices ($k > 0$)
        it suffices by \cref{ss-redkomp-einfach} that the following two sets coincide:
        \[
            \left\{ (C_0 \leq \cdots \leq C_k) \in ((M_m)^{(0)})^{k+1} \mid C_0 \cap \cdots \cap C_k
                \neq \emptyset \right\}, \]
        \[
            \left\{ (C_0 \leq \cdots \leq C_k) \in ((M_m)^{(0)})^{k+1} \mid 
                \tilde\pr_i(C_0) \cap \cdots \cap \tilde\pr_i(C_k) \neq \emptyset
                \quad \forall i\in\{1, \ldots, d\}\right\}.
        \]
        Let $(C_0 \leq \cdots \leq C_k) \in (W_{2^d})^{(0)}$. 
        The components $C_0, \ldots, C_k$ intersect iff they are contained in a common chart,
        which means they have to be totally ordered by the product order 
        defined on $W_{2^d}^{(0)}$. This is in turn equivalent to
        $\tilde\pr_i(C_0) \cap \cdots \cap \tilde\pr_i(C_k) \neq \emptyset$
        for all $i$.
    \end{par}
    \begin{par}
        The general case, where $X$ is any regular strict semi-stable curve, 
        can be reduced to the local case above:
        The property of $W(X,\leq,d)$ to be regular strict semi-stable
        can be checked locally. 
        By \cref{redkomp-lok-ber} also the simplicial reduction set can be determined in an open covering.
        Thus we may assume by \cref{redkomp-std-ubd} that there is a smooth map $X \to L$.
        By the commutativity of blow-ups with flat base change (\cite[Prop 8.1.12]{liu})
        it suffices to deal with the case $X=L$.
        For details see \cite[Lemma 2.11]{me}.
    \end{par}
\end{proof}

\begin{par}
    From \cref{de-lokblup-lem}(iv) we immediately deduce:
\end{par}
\begin{kor}
    The centres of all blow-up morphisms in \cref{desi-prodkomp} are scheme-theoretic intersections of
    Cartier divisors.
\end{kor}




\section{Intersection Numbers for Products} 
\label{kap-schnitt}

\begin{par}
    After the construction of regular strict semi-stable models on products of semi-stable
    curves, 
    we use the combinatorial structure of the reduction set
    to calculate intersection numbers.
    Let $S$ be the spectrum of a discrete valuation ring with algebraically closed
    residue field
    and $X$ a regular strict semi-stable $S$-curve with a total
    ordering $\leq$ on $X^{(0)}$.
    The reduction graph is denoted by $\Gamma:=\Gamma(X)$.
    For a chosen $d \in \IN$ we examine the product model
    $W:=W(X,\leq,d)$ of $(X_\eta)^d$ constructed according to \cref{desi-prodkomp}.
    By \cref{desi-prodkomp-ex} the simplicial reduction set is
    $\RK(W)=\Gamma^d$.
    Some essential relations in the Chow ring $\CH_{W_s}(W)$
    depend only on the combinatorial structure of $\RK(W)$.
    We define a ring $\KC(\Gamma^d)$ called combinatorial Chow ring, 
    which encodes these relations. 
    A moving lemma in $\KC(\Gamma^d)$ enables us to calculate intersection numbers
    in a more localised way. The intersection numbers in $\KC(\Gamma^d)$ for any graph
    $\Gamma$ can therefore be calculated only with the knowledge of $\KC(I^d)$,
    where $I$ is the graph which consists of exactly one edge.
    For such a graph we can eventually give explicit calculations of intersection
    numbers in the case $d=2$ and $d=3$.
\end{par}

\subsection{Intersection Theory on Regular Schemes} 

\begin{par}
    For our calculations we need only the basic facts of intersection theory, i.e.,
    intersection with Cartier divisors. This is covered in \cite[Chapter 1,2]{fulton} if
    one applies the generalizations from Chapter 20. 
    We give a short explanation of the notation, recall the basic facts 
    and deduce a lemma considering ramified base change.
\end{par}

\begin{defn}
    Let $S$ be a regular scheme and $X$ be a $S$-scheme of finite type
    with structure morphism $\varphi: X \to S$.
    \begin{enumerate}[(i)]
        \item
            If $X$ is irreducible with generic point $\eta_X$ we define the 
            \emph{relative dimension} of $X$ over $S$ by:
            \[
                \dim_S(X) := \trdeg(\kappa(\eta_X) / \kappa(\varphi(\eta_X))) 
                - \dim(\Oo_{S, \varphi(\eta_X)}).
            \]
        \item
            If $X$ is any $S$-scheme of finite type it is called \
            \emph{relative equidimensional of dimension $d$}, 
            if for each irreducible component $V$ the equation
            \[
                \dim_S(V) = d
            \]
            holds. 
            In this case we call $d$ the relative dimension of $X$ with respect to $S$.
    \end{enumerate}
\end{defn}

\begin{defn}
    \label{schnitt-chow-defn}
    Let $S$ be a regular Noetherian scheme and $X$ a regular $S$-scheme of finite type,
    which is relative equidimensional of dimension $d=\dim_S(X)$.
    Let $Y \subseteq X$ be a closed subset and $p \in \IN$ be a number.
    We denote by $\CH^p_Y(X)$ the $p$-th Chow group with support in $Y$, i.e., in 
    the notation of \cite{fulton}
    \[
        \CH^p_Y(X) := A_{d - p}(Y).
    \]
\end{defn}

\begin{bem}
    Let $X,X'$ be regular relative equidimensional schemes and $f: X' \to X$ a morphism.
    Let $Z \subseteq X, Z' \subseteq X'$ be closed subschemes.
    \begin{enumerate}[(i)]
        \item
            If $f$ is flat, it induces a morphism
            $f^*: \CH_Z^p(X) \to \CH_{f^{-1}(Z)}^p(X')$ for each $p \in \IN$. 
        \item
            If $f$ is any morphism, there exists a morphism
            $f^*: \CH_Z^1(X) \to \CH_{f^{-1}(Z)}^1(X')$. 
        \item
            If $f$ is proper, it induces a morphism
            $f_*: \CH_Z^p(X') \to \CH_{f(Z')}^{p-d}(X)$, where $d$ is the relative dimension
            of $X'$ over $X$.
    \end{enumerate}
    All these constructions are functorial.
\end{bem}
\begin{proof}
    These are basic facts \cite[1.7, 2.2, 1.4]{fulton}.  
    For the second claim, note that $\CH^1_Y(X)$ coincides with the pseudo-divisors on $X$
    with support $Y$.
\end{proof}
\begin{bem}
    Fulton defines an intersection product for divisors using (ii): 
    Let $X$ be a regular scheme. Then there is the intersection product
    \[
        \cdot: \CH^1_{Y}(X) \otimes \CH^p_{Z}(X) \to \CH^{p+1}_{Y \cap Z}(X)
    \]
    defined for each divisor $D \in \CH^1_Y(X)$ and each integral closed subscheme $V$ by 
    \[
        D \cdot [V] = [j^* D]
    \]
    where $j: V \to X$ denotes the canonical inclusion.
\end{bem}

\begin{par}
    We will only consider elements of $\CH^p_{X_s}(X)$ which are products of elements from
    $\CH^1_{X_s}(X)$.
    \cite[Prop 2.3 (d)]{fulton} suggests the following generalization of the pull-back:
\end{par}
\begin{defn}
    Let $\alpha = D_1 \cdot \ldots \cdot D_p \in \CH^p_{X_s}(X)$ be the product of Cartier
    divisors $D_1, \ldots D_p \in \CH^1_{X_s}(X)$ and $f: X' \to X$ any morphism. 
    Then $f^* \alpha$ is defined as
    \[
        f^* \alpha = (f^* D_1) \cdot \ldots \cdot (f^* D_d).
    \]
\end{defn}

\begin{par}
    Let $S=\spec{R}$ be the spectrum of a complete discrete valuation ring with algebraic
    closed residue field and $X$ a regular proper $S$-scheme. 
    By a construction similar to \cite[Def 1.4]{fulton} we define a local degree 
    for cycles with support in the special fibre $X_s$. 
    For this definition we remark that the special fibre of $S$ consists only of one point
    $\{s\}$ which has codimension 1, therefore we have a canonical isomorphism 
    $\CH^1_{\{s\}}(\spec{R}) \simeq \IZ$.
\end{par}
\begin{defn}
    \label{schnitt-deg}
    Let $X$ be a flat proper $S$-scheme of relative dimension $d$ with structure morphism
    $f: X \to S$.
    The morphism
    \[
        \ldeg_X:=f_*: \CH^{d+1}_{X_s}(X) \to \CH^1_{\{s\}}(S) \simeq \IZ
    \]
    is called \emph{local degree}.
\end{defn}
\begin{par}
    We are now able to describe the behaviour of the intersection product
    at ramified base change: 
    Let $K_n/K$ be an algebraic extension of degree $n$ of $K=\Quot(R)$, 
    let $R_n$ be the ring of integers in $K_n$
    and $S_n:=\spec{R_n}$ be the spectrum of $R_n$.
    Since the structure morphism $g: S_n \to S$ is proper, 
    we can consider the direct image 
    $g_*: \CH(S_n) \to \CH(S)$.
    This enables us to compare intersection multiplicities on $S$-schemes
    and $S_n$-schemes:
\end{par}
\begin{lem}
    \label{schnitt-deg-bw-lem}
    Let $W$ be a integral flat $S$-scheme and $W_n$ an integral $S_n$-scheme 
    with a proper $S$-morphism $\varphi: W_n \to W$.
    Assume that $\varphi \mid_{(W_n)_\eta}$ is flat and
    $(W_n)_\eta = W_\eta \times_{S_\eta} (S_n)_\eta$ holds.
    Then for each $\alpha=D_1 \cdots \ldots \cdot D_d \in \CH^d_{W_s}(W)$ 
    with $D_1, \ldots, D_d \in \CH^1_{W_s}(W)$ the equation
    \[
        \varphi_*(\varphi^*(\alpha)) = n \alpha
    \]
    holds in $\CH_{W_s}^d(W)$.
    If furthermore $W$ is a proper $S$-scheme and $W_n$ a proper $S_n$-scheme,
    then for $d=\dim W$ and each $\alpha \in \CH_{W_s}^{d}(W)$ we get the equation
    \[
        \ldeg_W(\alpha) = n \ldeg_{W_n}(\varphi^*(\alpha)).
    \]
\end{lem}
\begin{proof}
    \begin{par}
        We first show $\varphi_*([W_n]) = n[W]$.
        The image $\varphi(W_n)$ is irreducible, hence $\varphi_*([W_n])$ is a multiple
        of $[W]$. 
        To determine the multiplicity we restrict ourself to the generic fibre and
        consider the Cartesian square
        \[ 
            \begin{CD}
                (W_n)_{\eta_n} @>f'>> (S_n)_{\eta_n} \\
                @V\varphi VV     @VgVV \\
                W_\eta  @>f>> S_\eta.
            \end{CD}
        \]
        Using $g_*([(S_n)_{\eta_n}]) = n[S_\eta]$ we get
        \[ 
            \varphi_*[(W_n)_{\eta_n}] = \varphi_*( f'^*[(S_n)_{\eta_n}]) 
            = f^*(g_*([(S_n)_{\eta_n}])) 
            = n f^*([S_\eta]) = n[W_\eta]. 
        \]
    \end{par}
    \begin{par}
        Let now $\alpha=D_1\cdot\ldots\cdot D_d \in \CH_{W_s}(W)$ be arbitrary.
        Then the claim is proven using the projection formula \cite[Prop 2.3 (c)]{fulton}:
        \[ \varphi_*(\varphi^*(\alpha)) = \varphi_*( \varphi^*(\alpha) \cdot [W_n]) =
            \alpha \cdot \varphi_*([W_n]) = n \alpha. \]
    \end{par}
    \begin{par}
        For the second claim consider the commutative diagram
        \[ 
            \begin{CD}
                (W_n) @>f'>> (S_n) \\
                @V\varphi VV     @VgVV \\
                W  @>f>> S,
            \end{CD}
        \]
        where $f$ and $f'$ are proper.
        By definition of the degree map we get
        \begin{align*}
            f_*\varphi_*(\varphi^* \alpha) &= n f_* \alpha = n \ldeg_{W}(\alpha) [\{s\}]\\
            \intertext{and}
            f_*\varphi_*(\varphi^* \alpha) &= g_*f'_*(\varphi^* \alpha) 
            = \ldeg(\varphi^* \alpha) g_* ([\{s_n\}]) = \ldeg(\varphi^* \alpha) [\{s\}].
        \end{align*}
        The claim now follows by equating coefficients.
    \end{par}
\end{proof}


\subsection{The Chow Ring of Product Models} 
\label{kap-chow-semi}

\begin{par}
    We use the facts we recalled so far to study
    the Chow ring with support in the special fibre
    of regular strict semi-stable schemes.
    In particular we consider the product models 
    constructed in \cref{desi-prodkomp}.
    Let as usual $X$ be a regular strict semi-stable curve over $S$ with total 
    ordering $\leq$ on $X_s^{(0)}$. 
    We assume that the reduction graph $\Gamma(X)$ has no
    multiple edges. 
    Furthermore let $W:=W(X,\leq,d)$ denote the desingularization
    of $X^d$ constructed in \cref{desi-prodkomp}.
    We denote the composition
    of the desingularization map $W \to X^d$ with
    the projection on the $i$th factor
    by $\pr_i: W \to X$.
    Since $\Gamma(X)$ is without multiple simplices,
    this is also true for $\RK(W)$.
    We may deduce the following relations in the Chow ring of $W$:
\end{par}
\begin{prop}
    \label{chow-semi-eigtl}
    \begin{par}
        Let $C_1, \ldots C_k$ be pairwise different irreducible components 
        of $W_s$ with $C_1 \cap \cdots \cap C_k \neq \emptyset$.
        Then these components intersect properly and with multiplicity $1$, 
        i.e.,
        \[
            [C_1] \cdots \cdots \cdot [C_l] = [C_1 \cap \cdots \cap C_l].
        \]
    \end{par}
    \begin{par}
        If furthermore $l=d$ we have
        \[
            \ldeg([C_1] \cdot \cdots \cdot [C_l]) = 1.
        \]
    \end{par}
\end{prop}
\begin{proof}
    The components $C_1, \ldots C_k$ intersect properly according to \cref{rss-def} 
    and by \cref{ss-redkomp-einfach} the intersection $C_1 \cap \cdots \cap C_i$
    is irreducible.
    Using induction, it is enough to show
    $\chi^p(C_1 \cap \cdots \cap C_{l-1}, C_l) = 1$ in the generic point $p$
    of $C_1 \cap \cdots \cap C_l$ where $\chi^p$ denotes the intersection multiplicity of
    Serre. 
    This is a direct consequence of the following variant of the criterion for
    multiplicity 1 of Fulton \cite[7.2]{fulton}:
\end{proof}
\begin{lem}
    Let $Y,Z$ be integral regular closed subschemes of a regular scheme $X$ of codimension
    $p$ resp. $q$, which have proper intersection. If the scheme-theoretic intersection
    $Y \cap Z$ is reduced, then $Y \cap Z$ is regular and Serre's intersection
    multiplicity gives
    \[
        \chi^x(Y,Z) = 1
    \]
    for each generic point $x$ of an irreducible component of $Y \cap Z$.
\end{lem}
\begin{proof}
    \begin{par}
        Let $x$ be the generic point of an irreducible component of $Y \cap Z$. 
        Denote the local ring $\Oo_{X,x}$ with $A$ and the defining ideals 
        for $Y$ and $Z$ with $\IId_Y$ resp. $\IId_Z$. Since $Y$ is regular, $\IId_Y$
        is generated by a regular sequence $(y_1, \ldots, y_p)$ of length $p$; 
        likewise the ideal $\IId_Z$ is generated by a regular sequence $(z_1, \ldots, z_q)$.
        Since $x$ is a generic point, the elements $(y_1, \ldots, y_p, z_1, \ldots, z_q)$
        constitute an generating set for the maximum ideal $\mId_{X,x}$ of $\Oo_{X,x}$.
        Therefore $\Oo_{X,x}$ is itself regular and the generating set is a system of
        parameters.
    \end{par}
    \begin{par}
        By \cite[IV A Cor. 2]{localg} we may use the Koszul-complex to calculate
        \[ \Tor_i(A/(\IId_Y)_x, A/(\IId_Z)_x) \simeq H_i((y_1, \ldots, y_p), A/(\IId_Z)_x) \]
        and eventually by \cite[IV A Prop. 3]{localg}
        \[ H_i((y_1, \ldots, y_p), A/(\IId_Z)_x) = 0 \textrm{ for all }i\geq 1. \]
        Therefore we have
        \[ \chi^x(Y,Z) = \len(A/(\IId_Y)_x \otimes A/(\IId_Z)_x) = \len(A/\mId) = 1. \]
    \end{par}
\end{proof}

\begin{prop}
    \label{schnitt-raq-spfaser}
    For each irreducible component $C \in W_s^{(0)}$ of $W_s$ the equation
    \[ \left(\sum_{C' \in W_s^{(0)}} [C']\right) \cdot [C] = 0 \in \CH^\cdot_{W_s}(W) \]
    holds.
\end{prop}
\begin{proof}
    Since $W_s$ is reduced, we have 
    $\sum_{C' \in W_s^{(0)}}[C'] = \Div(\pi) \in \rat_{W_s}(W)$. 
    Therefore the intersection product with the cycle $[C]$ vanishes
    in the Chow ring $\CH^\cdot_{W_s}(W)$.
\end{proof}
\begin{prop}
    \label{schnitt-raq-prod}
    Let $C_1, C_2 \in \RK(W)_0$ be irreducible components of $W_s$. 
    If there is an $i \in \{1, \ldots d\}$ such that $\pr_i(C_1) \neq \pr_i(C_2)$,
    then the equation
    \[ [C_1] \cdot [C_2] \cdot \left( \sum_{\substack{C' \in \RK(W)_0,\\ \pr_i(C') = \pr_i(C_2)
                }} [C'] \right) = 0 \]
    holds in $\CH^{\cdot}_{W_s}(W)$.
\end{prop}
\begin{proof}
    We can assume that the intersection of $C_1$ and $C_2$ is non-empty.
    Since $\RK(X)$ has no multiple simplices and $\pr_i(C_1) \neq \pr_i(C_2)$,
    the projections $\pr_i(C_1)$ and $\pr_i(C_2)$ have to intersect in one double point.
    We denote this point by $p \in X$.
    The component $\pr_i(C_2)$ of $X_s$ is given by a Cartier divisor $\tilde C \in \cadiv(X)$ 
    and can be represented by a section $r \in \Gamma(U,\Oo_X)$ in an open neighbourhood
    $U$ of $p$.
    As the scheme $X$ is integral, we can continue $r$ as a rational function
    $r \in \Gamma(X, \Ok_X)$ and we have $\Div(r) = \tilde C + \tilde C_r$,
    where $\tilde C_r$ is a Cartier divisor with support outside of $U$.
    Using the pull-back we get a principal divisor on $W$,
    which splits into
    $\pr_i^{*}(\Div(r)) = \pr_i^{*}(\tilde C) + \pr_i^*(\tilde C_r)$. 
    Since $X_s$ is geometrically reduced, we have
    \[ [\pr_i^*(\tilde C)] = \sum_{\substack{C' \in \RK(W)_0,\\ 
                \pr_i(C') = \pr_i(C_2)}}[C'], \]
    while the residue divisor $\pr_i^*(\tilde C_r)$ has support outside of $p$ and 
    does not contribute to the intersection product $[C_1][C_2]$.
    Calculating in $\CH^{\cdot}_{W_s}(W)$ we finally get
    \[ 
        0 = [C_1][C_2][\pr_i^*(\Div(r))] = [C_1][C_2] 
        \sum_{\substack{C' \in \RK(W)_0,\\\pr_i(C') = \pr_i(C_2)}}[C']. 
    \]
\end{proof}


\subsection{A Moving Lemma in the Special Fibre} 
\begin{par}
    The rational equivalences considered in \cref{schnitt-raq-spfaser} and
    \cref{schnitt-raq-prod} depend only on the simplicial reduction set
    of the model $W$ and on the projections $\pr_i: \RK(W) \to \RK(X)$.
    This observation leads us to define a combinatorial Chow ring,
    which describes the part of the Chow ring 
    which is independent of the concrete model.
    On this Chow ring we can proof a moving lemma,
    which is used later to calculate intersection numbers.
\end{par}
\begin{par}
    For this section let $d \in \IN$ be an integer and $\Gamma$ a finite ordered graph
    without multiple edges, i.e. a simplicial set of dimension 1 without multiple
    simplices.
    We denote by $\Gamma^d$ the product of simplicial sets according to \cref{sk-sk-prod}
    and by $\pr_i: \Gamma^d \to \Gamma$ 
    the projection on the $i$th component.
    Furthermore we use the notation from \cref{sk-einfach-kompl},
    where $(\Gamma^d)_S \subseteq \mathcal{P}((\Gamma^d)_0)$
    denotes the family of all subsets,
    which occur as node set of a simplex in $\Gamma^d$, i.e., the set
    \[
        (\Gamma^d)_S = \{ \{\sigma(0), \ldots, \sigma(n)\} 
            \mid n \in \IN, \sigma \in (\Gamma^d)_n\}.
    \]
\end{par}
\begin{defn}
    \label{schnitt-chw}
    \begin{par}
        We denote with $Z(\Gamma^d)$ the polynomial ring
        $Z(\Gamma^d):= \IZ[C \mid C \in (\Gamma^d)_0]$ 
        generated by the 0-simplices. 
        It is supplied with the usual grading, which
        gives all generators $C \in (\Gamma^d)_0$ the degree 1.
    \end{par}
    \begin{par}
        We define a graded ideal $\mathrm{Rat}(\Gamma^d)$ on $Z(\Gamma^d)$ 
        generated by the polynomials
        \begin{align}
            \label{schnitt-chw1}
            C_1 \cdot \cdots \cdot C_k \quad & \textrm{ for } \{C_1, \ldots, C_k\} \not \in
            (\Gamma^d)_S, \\
            \label{schnitt-chw2}
            \Big(\sum_{C' \in (\Gamma^d)_0} C'\Big) C_1 \quad  & \forall C_1 \in (\Gamma^d)_0, \\
            \label{schnitt-chw3}
            \sum_{\substack{C' \in (\Gamma^d)_0\\ \pr_i(C')=\pr_i(C_2)}} C_1 C_2 C'\quad &
             \begin{aligned}
                \\
                \forall C_1,C_2 \in (\Gamma^d)_0, i \in \{1, \ldots, d\} \\
                \textrm{ with } \pr_i(C_1) \neq pr_i(C_2).
            \end{aligned}
        \end{align}
        We call $\mathrm{Rat}(\Gamma^d)$ the ideal of cycles rationally equivalent to zero.
    \end{par}
    \begin{par}
        The graded ring
        \[ \KC(\Gamma^d) := Z(\Gamma^d) / \mathrm{Rat}(\Gamma^d) \]
        is called \emph{combinatorial Chow ring}.
    \end{par}
\end{defn}
\begin{par}
    The combinatorial Chow ring has the following functoriality:
\end{par}
\begin{prop}
    \label{schnitt-C-funktor}
    Let $\Gamma$ and $\Gamma'$ be finite ordered graphs without multiple simplices
    and $f_1, \ldots f_d: \Gamma' \to \Gamma$ morphisms of graphs.
    Let $f=(f_1, \ldots f_d): (\Gamma')^d \to \Gamma^d$
    be the induced morphism on the products.
    Then we have a well-defined morphism of rings given by
    \[ f^*: Z(\Gamma^d) \to Z(\Gamma'^d), 
        C \mapsto \sum_{\substack{C' \in (\Gamma'^d)_0,\\ f(C') = C}}C'.
    \] 
    It induces an homomorphism of combinatorial Chow rings
    \[ f^*: \KC(\Gamma^d) \to \KC({\Gamma'}^d). \]
\end{prop}
\begin{proof}
    \begin{par}
        We have to show $f^*(\mathrm{Rat}(\Gamma^d)) \subseteq \mathrm{Rat}(\Gamma'^d)$
        and can restrict ourself to the generators 
        \cref{schnitt-chw1}, 
        \cref{schnitt-chw2} and
        \cref{schnitt-chw3}.
    \end{par}
    \begin{par}
        For each set $\alpha \in (\Gamma'^d)_S$ we have
        $f(\alpha) \in (\Gamma^d)_S$. 

        Consider a polynomial of type $\cref{schnitt-chw1}$,
        i.e., a product $\alpha = C_1 \cdot \cdots \cdot C_k$ 
        with $\{C_1, \cdots C_k\} \not \in (\Gamma^d)_S$.
        Then each monomial of $f^* \alpha$ has 
        the form $\tilde C_1 \cdot \ldots \cdot \tilde C_k$
        with $f(\tilde C_i) = C_i$ and since $f$ is a morphism of simplicial sets,
        $\{\tilde C_1, \cdots, \tilde C_k\} \not \in (\Gamma'^d)_S$ holds.
        This implies 
        $f^*(C_1 \cdots \cdot \cdots C_k) \in \mathrm{Rat}(\Gamma'^d)$.
    \end{par}
    \begin{par}
        Similarly, one can show the claim for elements of type \cref{schnitt-chw2} and
        \cref{schnitt-chw3} by means of the equations
        \begin{align*}
            f^*\Big(\sum_{C' \in ((\Gamma')^d)_0} C'\Big) &= \sum_{C \in (\Gamma^d)_0}C \\
            \intertext{and}
            f^*\Big(\sum_{\substack{C' \in ((\Gamma')^d)_0 \\ \pr_i(C') =
                        \pr_i(C'_2)}}C'\Big)
            &= \sum_{\substack{C \in (\Gamma^d)_0\\ \pr_i(C') = \pr_i(f(C'_2))}}C.
        \end{align*}
    \end{par}
\end{proof}
\begin{par}
    With the definition of the combinatorial Chow ring we can sum up our knowledge about
    the Chow ring of a product model $W=W(X,<,d)$ by:
\end{par}

\begin{prop}
    \label{schnitt-vergl-chow}
    Let $d \in \IN$, $X$ a regular strict semi-stable $S$-curve and $<$ be
    a total ordering on $X^{(0)}$.
    We denote by $W=W(X,<,d)$ the model of $(X_\eta)^d$ 
    constructed in \cref{desi-prodkomp}.
    Then there is a morphism of graded rings defined by
    \[
        \varphi_W: \KC(\Gamma(X)^d) \to \CH^*_{W_s}(W), [C] \mapsto [C].
    \]
\end{prop}
\begin{proof}
    The relation \cref{schnitt-chw1} is trivial, 
    relations \cref{schnitt-chw2} and \cref{schnitt-chw3} follow from 
    \cref{schnitt-raq-spfaser} and \cref{schnitt-raq-prod}.
\end{proof}

\begin{par}
    Let us focus on the main result of this section, a moving lemma on
    $\KC(\Gamma^d)$.
    For this we use a definition of proper cycles in $Z(\Gamma^d)$ similar to the one
    in intersection theory.
\end{par}
\begin{defn}
    A monomial $C_1 \cdots \cdots \cdot C_k \in Z(\Gamma^d)$ is called
    \emph{proper}, 
    if the vertices $C_i \in \Gamma_0$ are pairwise different.
    An arbitrary element $\alpha \in Z(\Gamma^d)$ is called \emph{proper},
    if it is a sum of proper monomials.
\end{defn}
\begin{satz}
    \label{schnitt-moving}
    Let $\Gamma$ be a connected finite graph without multiple simplices.
    Then the group $\KC^{k}(\Gamma^d)$ of the $k$-cycle classes
    is generated by the cycle classes of proper monomials
    \[
        \{ C_1 \cdot \ldots \cdot C_{k} \mid C_1, \ldots, C_{k} \in Z(\Gamma^d) 
            \textrm{ pairwise different} \}. 
    \]
\end{satz}
\begin{par}
    Note that the special fibre $X_s$ of a regular strict semi-stable $S$-curve
    is connected and so is the simplicial reduction set $\RK(X)$.
\end{par}

\begin{par}
    Before we approach the proof, we define a useful decomposition of $\Gamma^d$:
\end{par}
\begin{bem}
    \label{schnitt-graph-zerleg}
    Let $\Gamma$ be a graph without multiple simplices. According to \cref{sk-kolim-simpl}
    we identify the 1-simplices $\gamma_1 \in \Gamma_1$ with
    morphisms $i_{\gamma_1}: \Delta[1] \to \Gamma$.
    Since $\Gamma$ is without multiple simplices, $i_{\gamma_1}$ is injective 
    for each non-degenerate 1-simplex $\gamma_1$.
    Let now $\gamma:=(\gamma_1, \ldots \gamma_d) \in (\Gamma_1^{\mathrm{nd}})^d$ 
    be a $d$-tuple of 1-simplices.
    The product 
    \[
        i_\gamma := (i_{\gamma_1} \times \cdots \times \cdots \times i_{\gamma_d}):
        I^d \to \Gamma^d
    \]
    is injective as well and denoted by $i_\gamma$. 
    The set of all $i_\gamma$ gives a covering of $\Gamma^d$:
\end{bem}
\begin{prop}
    \label{schnitt-graph-quad}
    Let $\Gamma$ be a finite connected graph without multiple simplices,
    which contains at least two vertices. 
    Then the images of $i_\gamma$ for $\gamma \in (\Gamma_1^{\textrm{nd}})^d$ 
    yield a covering of $\Gamma^d$, 
    the \emph{covering by standard cubes}.
    An arbitrary $k$-simplex $\sigma \in (\Gamma^d)_k$ 
    is in the image of $i_\gamma$, iff
    all vertices of $\sigma$ are contained in the image of $i_\gamma$.
    If $\sigma \in (\Gamma^d)^{\textrm{nd}}_d$ is a non-degenerate $d$-simplex,
    there is exactly one $\gamma \in (\Gamma_1^{\textrm{nd}})^d$
    such that $\sigma$ is contained in the image of $i_\gamma$.
\end{prop}
\begin{proof}
    elementary (see \cite{me}) 
\end{proof}

\begin{par}
    Let $C_1 \cdots \cdot \cdots C_k$ be a monomial; for the proof 
    of \cref{schnitt-moving} we call the cardinality 
    $\card\{C_1, \ldots C_k\}$
    the size of $C_1 \cdots \cdot \cdots C_k$.
    A monomial in $\KC^{k}(\Gamma^d)$ of size $k$ is obviously 
    a proper monomial.
    The proof is then carried out by induction on the maximum size of the monomials
    involved. 
    We may reduce it to the following lemma:
\end{par}

\begin{lem}
    \label{schnitt-moving-step}
    Let $\alpha=C_1 \cdot\cdots\cdot C_k \in Z^k(\Gamma^d)$ be a monomial of degree $k$
    and size $l < k$.
    Then there exists an element $\alpha' \in Z^k(\Gamma^d)$
    which consists of monomials of size $l+1$ or greater.
\end{lem}

\begin{par}
    The main ingredient for the proof is the following lemma, a kind of reduction to the
    standard cube:
\end{par}

\begin{lem}
    \label{schnitt-schieb}
    Let $\Gamma$ be a connected finite graph without multiple simplices,
    $\gamma \in (\Gamma_1^{\textrm{nd}})^d$ and
    $i_\gamma: I^d \to \Gamma^d$ the associated embedding 
    according to \cref{schnitt-graph-zerleg}.
    We endow the vertices of the standard cube $I^d$ 
    with the product ordering and denote it by $\leq$.
    Let $C_1, C_2 \in (I^d)_0$ be vertices of the standard cube
    with $C_1 < C_2$.
    Then the following holds:
    \begin{enumerate}[(i)]
        \item
            There is a 1-cycle $\beta \in Z^1(\Gamma^d)$
            such that
            \[ i_\gamma(C_1) i_\gamma(C_2)^2 - \beta i_\gamma(C_1)i_\gamma(C_2) \in
                \Rat(\Gamma^d) \]
            holds and we have
            \[ i_\gamma^*(\beta) = \sum_{i=1}^n [E_i] \]
            for a finite number of elements $E_i \in I^d_0$
            with $E_i > C_1$ and $E_i \neq C_2$.
        \item
            There is a 1-cycle $\beta' \in Z^1(\Gamma^d)$
            such that
            \[ i_\gamma(C_1)^2 i_\gamma(C_2) - \beta' i_\gamma(C_1)i_\gamma(C_2) \in
                \Rat(\Gamma^d) \]
            holds and we have
            \[ i_\gamma^*(\beta') = \sum_{i=1}^n [E'_i] \]
            for a finite number of elements $E'_i \in I^d_0$
            with $E'_i < C_2$ and $E'_i \neq C_1$.
    \end{enumerate}
\end{lem}
\begin{proof}
    \begin{par}
        Since $C_1 < C_2$, there exists $j \in \{1, \ldots, d\}$
        such that $\pr_j(C_1) < \pr_j(C_2)$.
        We set
        \[
            \tilde\beta := -\sum_{\substack{
                    C' \in (\Gamma^d)_0 \setminus \{i_\gamma(C_2)\} \\
                    \pr_j(C') = \pr_j(i_\gamma(C_2))
                }}[C']
        \]
        and get according to \cref{schnitt-chw3} in \cref{schnitt-chw}
        \begin{equation}
            \label{schnitt-schieb-gl}
            i_\gamma(C_1)i_\gamma(C_2)^2 - \tilde\beta i_\gamma(C_1)i_\gamma(C_2) \in
            \Rat(\Gamma^d).
        \end{equation}
        Applying the injection $i_\gamma$ we get
        \[
            i_\gamma^*(\tilde\beta) = -\sum_{\substack{
                    C \in (I^d)_0 \setminus \{C_2\} \\
                    \pr_j(C') = \pr_j(C_2)
                }}C.
        \]
        Each of the elements $C \in (I^d)_0 \setminus \{C_2\}$
        with $\pr_j(C) = \pr_j(C_2)$ 
        is either bigger than $C_1$ or not comparable with $C_1$.
        If $C$ is not comparable with $C_1$, then
        $C_1 C \in \Rat(I^d)$ holds
        and according to \cref{schnitt-graph-quad} also
        $i_\gamma(C_1) i_\gamma(C) \in \Rat(\Gamma^d)$.
        Therefore we can remove these elements from $\tilde\beta$
        without changing the equation \cref{schnitt-schieb-gl}.
        The cycle
        \[
            \beta := \sum_{\substack{
                    C' \in (\Gamma^d)_0 \setminus \{i_\gamma(C_2)\} \\
                    \pr_j(C') = \pr_j(i_\gamma(C_2)) \\
                    C' \not \in \{ i_\gamma(C) \mid C \not > C_1\}
                }}C'
        \]
        satisfies the claim. 
    \end{par}
    \begin{par}
        The proof of (b) is done analogously.
    \end{par}
\end{proof}

\begin{proof}[Proof of \cref{schnitt-moving-step}]
    \begin{par}
        If $\Gamma$ consists only of one vertex,
        the statement is trivial.
        In this case $\Gamma^d$ has only one vertex
        and according to \cref{schnitt-chw2} in \cref{schnitt-chw}
        each monomial of degree $k \geq 2$ is in $\mathrm{Rat}(\Gamma^d)$.
    \end{par}
    \begin{par}
        Let 
        $\alpha = \tilde C_1^{a_1} \cdot \cdots \cdot \tilde C_l^{a_l} \in Z^k(\Gamma^d)$
        be a monomial of degree $k$ where $C_1, \ldots C_l$ are different vertices of
        $(\Gamma^d)_0$.
        If $l=1$, then a single application of \cref{schnitt-chw2} yields
        an equivalent cycle of size $2$, i.e.,
        consisting of monomials with at least two different factors.
        Thus we may assume from now on that $l \geq 2$ and $\Gamma$ is connected with
        at least two vertices.
    \end{par}
    \begin{par}
        If $l \geq 2$ we can assume
        that $\{\tilde C_1, \ldots \tilde C_l\}$ is a simplex in $\Gamma^d$,
        otherwise we would have $\alpha=0$.
        According to \cref{schnitt-graph-zerleg} there exists an embedding
        of the standard cube 
        $i_\gamma: I^d \to \Gamma^d$
        such that
        $\tilde C_1, \ldots, \tilde C_l$ are in its image.
        We denote the preimages with $C_1, \ldots C_l$. 
        Since theses preimages constitute a simplex in $I^d$,
        we may assume that $C_1 < \ldots < C_l$ holds
        (according to the product ordering in $I^d$).
        We get
        \[ 
            i_\gamma^*\alpha = C_1^{a_1} \cdot\cdots\cdot C_l^{a_l}.
        \]
        Let us denote by $j(\alpha)$ the minimal index
        \[
            j = j(\alpha) = \min\{ j' \in \{1, \ldots l\} \mid a_{j(\alpha)} \geq 2\}.
        \]
        where a component occurs twice.
        It suffices to show
        that there is a 1-cycle
        \[ 
            \beta=\sum_{\tilde C' \in (\Gamma^d)_0} b_{C'} C' \in Z^1(\Gamma^d), 
            \quad b_{C'} \in \IZ
        \] 
        such that
        \[
            \alpha - 
            \beta C_1^{a_1} \cdot \ldots \cdot C_j^{a_j -1} \cdot \ldots \cdot C_l^{a_l}
            \in \Rat(\Gamma^d)
        \]
        and $b_{\tilde C_i}=0$ holds for all $i \in \{1, \ldots, l\}$.
        We show this once again by induction on $j(\alpha)$:
    \end{par}
    \begin{par}
        If $j=j(\alpha) < l$, we choose $\beta'$ according to \cref{schnitt-schieb} (ii)
        such that
        \[ 
            C_j^2C_{j+1} - \beta' C_j C_{j+1}  \in \Rat(\Gamma^d)
        \]
        with $i^* \beta' = \sum E_i$ for $C_j \neq E_i < C_{j+1}$.
        We dissect the element
        \[
            \beta C_1 \cdot\cdots\cdot C_j^{a_j-1} C_{j+1}^{a_{j+1}} \cdot\cdots\cdot C_l^{a_l} 
        \]
        into monomials $\sum_{i} \alpha_i$.
        Each monomial $\alpha_i$ either contains an additional factor and has therefore a
        bigger size than $\alpha$ or is of the form
        $C_1 \cdot\cdots\cdot C_{j-1} C_j^{\tilde a_j} \cdot\cdots\cdot C_l^{\tilde a_l}$
        with $\tilde a_j < a_j$.
        If we execute the same substitution with the latter monomials,
        we finally get $\tilde a_j=1$ and therefore
        $j(\alpha_i) < j(\alpha)$.
        The proposition follows then by induction.
    \end{par}
    \begin{par}
        On the other hand if $j=j(\alpha) = l$, we proceed analogous using
        \cref{schnitt-schieb} (i).
        We choose $\beta' \in Z^1(\Gamma, <, d)$
        such that
        \[ 
            C_{j-1}C_j^2 - \beta' C_{j-1}C_j \in \Rat(\Gamma^d)
        \]
        with $i^* \beta' = \sum_i E_i$ for $C_j \neq E_i > C_{j-1}$ for all $i$ holds.
        Then each monomial of the term
        \[ \beta C_1 \cdot\cdots\cdot C_j^{a_j-1} C_{j+1}^{a_{j+1}} \cdot\cdots\cdot C_l^{a_l} \]
        has a bigger size than $\alpha$.
    \end{par}
\end{proof}

\subsection{The Local Degree Map} 

\begin{par}
    With the help of the moving lemma we can define a degree map on the combinatorial
    Chow ring, which is compatible with the degree map of \cref{schnitt-deg}.
    We define it first for the standard $1$-simplex 
    $I:=\Delta[1]$.
    This graph can be obtained as the reduction set of the proper regular strict semi-stable scheme
    $\bar L := \Proj{R[z_0,z_1,t]/(z_0z_1 - \pi t^2})$ (it is the projective completion of
    the standard scheme $L$ considered in \cref{desi-aufprod}).
    Denote by $\bar M$ the desingularization of $\bar L^d$ according to
    \cref{desi-prodkomp}.
    Since $\bar M$ is a proper $S$-scheme, 
    there exists a local degree map
    \[
        \ldeg_{\bar M}: \CH^{d+1}_{\bar M_s}(\bar M) \to \IZ
    \]
    according to \cref{schnitt-deg}. 
    Furthermore we have $\RK(\bar M) = I^d$ and therefore a morphism of graded rings
    $\varphi_{\bar M}: \KC(I^d) \to \CH^{\cdot}_{\bar M_s}(\bar M)$
    by \cref{schnitt-vergl-chow}.
\end{par}
\begin{defn}
    We call the morphism of $\IZ$-modules
    \[
        \ldeg_{(I^d)} := \ldeg_{\bar M} \circ \varphi_{\bar M}: \KC(I^d) \to \IZ
    \]
    \emph{local degree map}.
\end{defn}
\begin{par}
    To define a local degree map for products of arbitrary graphs, we use
    the decomposition into standard cubes from \cref{schnitt-graph-quad}:
\end{par}
\begin{defn}
    \label{schnitt-grad-defn}
    Let $\gamma$ be a finite graph and $d \in \IN$. 
    For each $\gamma = (\gamma_1, \ldots \gamma_d) \in (\Gamma_1^{\textrm{nd}})^d$ denote 
    by $i_\gamma: I^d \to \Gamma^d$ the associated embedding of the standard cube 
    as defined in \cref{schnitt-graph-zerleg}
    and by $i_\gamma^*$ the respective morphism of graded rings
    $i_\gamma^*: \KC(\Gamma^d) \to \KC(I^d)$ as defined in \cref{schnitt-C-funktor}.
    The \emph{local degree map}
    $\ldeg_{\Gamma^d}: \KC(\Gamma^d) \to \IZ$ is then defined by
    \[ \ldeg_{(\Gamma,<,d)} := \sum_{\gamma \in (\Gamma_1^{\textrm{nd}})^d} \ldeg_{\bar M} \circ\; i_\gamma^*. \]
\end{defn}
\begin{prop}
    \label{schnitt-grad-vergleich}
    Let $X$ be a regular strict semi-stable curve over $S$ with total ordering $<$
    on $X^{(0)}$ and $W:=W(X,<,d)$ the associated model of the $d$\nobreakdash-fold
    product given by \cref{desi-prodkomp}.
    We denote the morphism between the Chow ring and the combinatorial Chow ring
    by
    $\varphi_W: \KC(\Gamma(X)^d) \to \CH_{W_s}(W)$.
    Then the local degree map of $\KC(\Gamma(X)^d)$ 
    coincides with the degree map of $\CH_{W_s}$, i.e.,
    \[ \ldeg_{(\Gamma(X)^d)} = \ldeg_{W} \circ\; \varphi_W. \]
\end{prop}
\begin{proof}
    By \cref{schnitt-moving} it suffices to show that the maps coincide for
    proper monomials.
    Let $C_0, \ldots C_d \in \RK(W)$ be different irreducible components of $W_s$
    with $\{C_0, \ldots C_d\} \in \RK_S(W)$,
    thus $C_0 \cap \cdots \cap C_d \neq \emptyset$.
    According to \cref{chow-semi-eigtl} all proper intersections in $\CH_{W_s}(W)$
    have multiplicity 1, therefore
    $\ldeg_W \circ\; \varphi_W(C_0 \cdot\cdots\cdot C_d) = 
    \ldeg_W ( [C_0] \cdot\cdots\cdot [C_d]) = 1$
    holds.
    To determine the left hand side, we examine the $d$-simplex 
    given by the vertices $(C_0, \ldots C_d)$. 
    Since this simplex is non-degenerate, there is 
    exactly one tuple $\gamma=(\gamma_1, \ldots \gamma_d) \in (\Gamma(X)^d)$ such that
    $i_\gamma^*:I^d \to \Gamma^d$ maps the monomial 
    $C_0 \cdot \cdots \cdot C_d$ to a proper monomial in $\KC(I^d)$ (see
    \cref{schnitt-graph-quad}).
    For each $\gamma' \in (\Gamma_1^{\textrm{nd}})^d$ with $\gamma \neq \gamma'$ one has
    $i_{\gamma'}^*(C_0 \cdot \cdots \cdot C_d) = 0$.
    Therefore 
    \begin{align*}
    \ldeg_{(\RK(X)^d)}(C_0 \cdot\cdots\cdot C_d) 
        &= \sum_{\tilde s \in \Gamma_1^d} \ldeg_{(I,<,d)}\circ\; i_{\tilde s}^*(C_0
        \cdot\cdots\cdot C_d)\\
        &= \ldeg_{(I,<,d)} \circ\; i_s (C_0 \cdot\cdots\cdot C_d) = 1. 
    \end{align*}
\end{proof}


\subsection{Explicit Calculations in $\KC(I^d)$} 

\label{schnitt-konkret}

\begin{par}
    By the preceding paragraph the calculation of $\KC(\Gamma^d)$ for arbitrary graphs
    can be reduced to the calculation of $\KC(I^d)$, the product of the 1-simplex.
    We introduce first a convenient notation for the elements of the Chow ring and
    introduce an alternative generator set of $\KC(I^d)$ using 
    a discrete fourier transform.
    For this alternative set of generators we deduce some relations.
\end{par}
\begin{par}
    We denote the vertices of $I$ as usual with $C_0$ and $C_1$ and endow them 
    with the ordering $C_0 < C_1$. For each vector $v=(v_1, \ldots v_d) \in \IF_2^d$
    let $C_v$ denote the vertex from $I^d$ with $\pr_i(C_v) = C_{v_i}$.
    Furthermore we call the vectors of the standard basis of $\IF_2^d$ as usual
    $e_1 := (1, 0, \ldots, 0), \ldots e_d:=(0, \ldots, 0, 1)$ and 
    set $E:=\{e_1, \ldots e_d\}$.
\end{par}
\begin{par}
    With this notation the set $\{ C_v \mid v \in \IF_2^d\}$ is 
    a generating set of 
    $\KC(I^d)_\IQ:=(\KC(I^d)) \otimes_\IZ \IQ$ 
    and the relations from \cref{schnitt-chw}
    can be written in the following form:
    \begin{align}
        C_v C_w & = 0 & & \parbox{7cm}{
            $\forall v,w \in \IF_2^d:
            \exists i,j \in \{1, \ldots, d\}$\\
            with
            $(v_i,v_j)=(w_j,w_i)=(1,0)$,
            } \label{schnitt-glb}\\
        \left( \sum_{w \in \IF_2^d} C_w \right) C_v & = 0 & & 
        \textrm{for all $v \in \IF_2^d$}, \label{schnitt-gla}\\
        C_v C_{v'} \left(\sum_{\substack{w \in \IF_2^d \\ w_i = v_i}} C_w\right) & = 0 & &
            \parbox{7cm}{
                for all $i \in \{1, \ldots, d\}$ and $v,v' \in \IF_2^d$ \\
                with $v_i \neq v'_i$.
        } \label{schnitt-glc} 
    \end{align}
\end{par}
\begin{bem}
    For relation \cref{schnitt-glb} it suffices to consider pairs
    $\{v,w\}$ of vectors: 
    If $C_1, \ldots C_k \in (I^d)_0$ is a tuple 
    with $\{C_1, \ldots C_k\} \not\in (I^d)_S$,
    then there exists a pair $i,j \in \{1, \ldots k\}$
    such that $\{C_i, C_j\} \not \in (I^d)_s$.
\end{bem}
\begin{par}
    For the calculation of the local degree it suffices by \cref{schnitt-moving}
    to consider proper intersections of $d+1$ elements,
    that is the product of $d+1$ pairwise different elements
    $C_{v_0}, \ldots, C_{v_d} \in \KC(I^d)$, 
    such that $\{ C_{v_0}, \ldots, C_{v_d}\} \in \KC(I^d)_S$.
    By the description of $\KC(I^d)$ in \cref{sk-prod-poset} 
    the last condition means that $v_0, \ldots, v_d$ is 
    (up to permutation) an ascending chain
    according to the product ordering. 
    Together with \cref{chow-semi-eigtl} 
    the local degree is uniquely defined by setting
    \begin{equation}
        \label{schnitt-konk-deg}
        \ldeg(C_{v_0} C_{v_1} \cdots C_{v_d}) = 1
    \end{equation}
    if $v_0 < v_1 < \ldots < v_d$ according to the product ordering.
\end{par}
%

\begin{defn}
    Let $d \in \IN$. For each vector $v \in \IF_2^d$ we denote by
    $F_v$ the element
    \[
        F_v := \sum_{w \in \IF_2^d}(-1)^{\angles{v,w}}C_w.
    \]
\end{defn}
\begin{bem}
    A straightforward computation in $\KC(I^d)_\IQ$ shows
    \[
        C_v = \frac{1}{2^d}\sum_{w \in \IF_2^d}(-1)^{\angles{v,w}} F_w.
    \]
    Therefore the set $\{F_v \mid v \in \IF_2^d\}$ is also a system of generators
    for $\KC(I^d)_\IQ$.
\end{bem}
\begin{par}
    We deduce some useful relations of the $F_v$:
\end{par}

\begin{prop}
    \label{schnitt-F}
    Let $v,v' \in \IF_2^d$ and $e,e' \in E$ be two base vectors of the standard basis.
    Then the following relations hold in $\KC(I^d)_\IQ$:
    \begin{align}
        F_0 F_v  &= 0, \label{schnitt-F-gl1} \\
        (F_{v+e+e'} - F_v)(F_{v'+e+e'} - F_{v'}) &= (F_{v+e}-F_{v+e'})(F_{v'+e}-F_{v'+e'}),
        \label{schnitt-F-gl2}\\
        F_e (F_v + F_{v+e})(F_v' - F_{v'+e})  &= 0. \label{schnitt-F-gl3}
    \end{align}
\end{prop}
\begin{proof}
    \begin{par}
        Equation \cref{schnitt-F-gl1} follows directly 
        from \cref{schnitt-gla} using $F_0 = \sum_{v \in \IF_2^d}C_v$.
    \end{par}
    \begin{par}
        For the proof of \cref{schnitt-F-gl2} choose $i,j \in \{1, \ldots d\}$
        such that $e=e_i, e'=e_j$ and denote by 
        $J \subseteq \IF_2^d$ the subset
        $J:= \{w \in \IF_2^d \mid w_i \neq w_j \}$.
        Then one has
        \begin{align*}
            (F_{v} - F_{v+e+e'}) & = 2 \sum_{w \in J} (-1)^{\angles{v,w}} C_w, \\
            (F_{v+e} - F_{v+e'}) & = 2 \sum_{w \in J} (-1)^{\angles{v,w}}(-1)^{\angles{e,w}} C_w.
        \end{align*}
        Therefore one calculates
        \begin{align*}
            & (F_{v} - F_{v+e+e'})(F_{v'} - F_{v'+e+e'}) 
            - (F_{v+e}- F_{v+e'})(F_{v'+e} - f_{v'+e'}) \\
            = &4 \sum_{w,w' \in J} (-1)^{\angles{v,w}+\angles{v',w'}} ( 1 
            - (-1)^{\angles{e,w}+\angles{e,w'}})C_wC_{w'} \\
            = &8 \sum_{\substack{w,w' \in J\\ w_i \neq w'_i}}
            (-1)^{\angles{v,w}+\angles{v',w'}}C_wC_{w'}
            = 0.
        \end{align*}
        The elements of the sum in the last equation vanish
        by \cref{schnitt-glb} and the relations
        $w_i \neq w'_i, w_j \neq w'_j$.
    \end{par}
    \begin{par}
        After all, equation \cref{schnitt-F-gl3} can be deduced 
        from 
        \begin{align}
            (F_0 + F_{e_i})   &= \sum_{\substack{w \in \IF_2^d\\ w_i = 0}} C_w,
            \label{schnitt-F-hilf1}\\
            (F_v + F_{v+e_i}) &= \sum_{\substack{w \in \IF_2^d\\ w_i = 0}}
            (-1)^{\angles{v,w}}C_w
            \textrm{ and} \nonumber \\
            (F_v - F_{v+e_i}) &= \sum_{\substack{w \in \IF_2^d\\ w_i = 1}}
            (-1)^{\angles{v,w}}C_w. \nonumber
        \end{align}
        By \cref{schnitt-glc} one has
        \[ (F_v + F_{v+e_i})(F_w-F_{w+e_i})(F_0 + F_{e_i}) = 0 \]
        and by $F_0 F_{v'} = 0$ for each $v' \in \IF_2^d$ 
        one gets the claim.
    \end{par}
\end{proof}
\begin{bem}
    \label{schnitt-F0-weg}
    Since $\{F_v \mid v \in \IF_2^d\}$ is a system of generators for $\KC(I^d)_\IQ$,
    equation \cref{schnitt-F-gl1} already implies 
    \[ F_0 \cdot \alpha = 0 \]
    for each element $\alpha \in \KC(I^d)_\IQ$ with positive degree.
\end{bem}

\begin{par}
    In the following situation the degree can be calculated generally:
\end{par}
\begin{prop}
    \label{schnitt-F-deg}
    In $\KC(I^d)_\IQ$ the equation
    \[ \ldeg(F_{v} \prod_{i=1}^d F_{e_i}) = \begin{cases}
            (-4)^{d} & \text{if $v=(1,\ldots, 1)$},\\
                0 & \text{otherwise}
        \end{cases}
    \]
    holds.
\end{prop}
\begin{proof}
    \begin{par}
        For the proof we use the decomposition
        \[ F_{v} = \sum_{w \in \IF_2^d} (-1)^{\angles{v,w}} C_w \]
        and calculate for each vector $w \in \IF_2^d$ the value 
        of
        $\ldeg\big(C_w \prod_{i=1}^d F_{e_i}\big)$.
    \end{par}
    \begin{par}
        For each $t \in \{0,1\}$ denote by
        $\alpha^t(w)$ the product
        \[ 
            \alpha^t(w):= \prod_{\substack{i \in \{1, \ldots, d\} \\ w_i = t}} F_{e_i}.
        \]
        Obviously we have 
        $\prod_{i=1}^d F_{e_i} = \alpha^0(w)\alpha^1(w)$.
    \end{par}
    \begin{par}
        We will proof in a moment for each $j$ with $w_j \neq t$ the recursion
        \begin{equation}
            \label{schnitt-F-deg-eq}
            C_w \alpha^t(w) = 2(-1)^t C_w C_{w+e_j} \alpha^t(w + e_j).
        \end{equation}
        Let us first show that this is enough to proof the claim:
        We may take a maximal increasing chain of vectors 
        $(0,\ldots,0)=w^{(0)} < \cdots < w^{(d)} = (1,\ldots,1) \in \IF_2^d$ 
        which contains $w$.
        This means $w^{(k)} = w$ for $k = \angles{w, (1, \ldots,1)}$.
        Multiple application of \cref{schnitt-F-deg-eq} shows 
        \[ C_w \alpha^0(w) \alpha^1(w) 
            = 2^d(-1)^{d-|w|} C_{w^{(0)}} \cdot\cdots\cdot C_{w^{(d)}}.
        \]
        We can therefore calculate the degree by \cref{schnitt-konk-deg}
        \[ \ldeg\left(C_w \prod_{i=1}^d C_{e_i} \right) = 2^d (-1)^{d-|w|} 
            = (-2)^d (-1)^{\angles{w,(1,\ldots, 1)}}. \]
        The claim now follows:
        \begin{align*}
            \ldeg\left(F_v \prod_{i=1}^d F_{e_i} \right) 
            & = \sum_{w \in \IF_2^d} (-1)^{\angles{v,w}} 
            \ldeg\left(C_w \prod_{i=1}^d C_{e_i}\right) \\
            & = \sum_{w \in \IF_2^d} (-1)^{\angles{v,w}}
            (-2)^d (-1)^{\angles{v,(1,\ldots, 1)}}  \\
            & = (-4)^d \delta_{v,(1,\ldots, 1)}.
        \end{align*}
    \end{par}
    \begin{par}
        It only remains to show \cref{schnitt-F-deg-eq}.
        For this purpose we show that the difference
        \begin{equation}
            \label{schnitt-F-deg-differenz}
            C_w \alpha^t(w) - 2(-1)^t C_w C_{w+e_j} \alpha^t(w+e_j) 
        \end{equation}
        vanishes.
        According to the precondition $w_j \neq t$ we have
        $\alpha^t(w) = F_{e_j} \alpha^t(w+e_j)$ and by \cref{schnitt-F0-weg}
        \[ 
            C_w F_{e_j} = (-1)^t C_w (F_0 + (-1)^t F_{e_j})
            = (-1)^t C_w \left( 
                2\sum_{\substack{w' \in \IF_2^d \\ w'_j = t}} C_{w'}
            \right).
        \]
        The second equation holds due to 
        equation \cref{schnitt-F-hilf1} in \cref{schnitt-F}.
        We conclude for the difference term \cref{schnitt-F-deg-differenz}
        \begin{align*}
            & C_w\alpha^t(w) - 2(-1)^t C_w C_{w+e_j} \alpha^t(w+e_j) \\
            = &C_wF_{e_j}\alpha^t(w+e_j) - 2(-1)^tC_wC_{w+e_j}\alpha^t(w+e_j) \\
            = &2(-1)^t C_w \alpha^t(w+e_j) \left(
                \sum_{\substack{w' \in \IF_2^d \setminus \{w+e_j\} \\
                        w'_j = t}} C_{w'}
            \right)
        \end{align*}
        and therefore it suffices to show 
        \[ C_{w'} C_w \alpha^t(w+e_j) = 0 \]
        for each $w' \in \IF_2^d$ with $w'_j = t$ and $w' \neq w+e_j$.
        Since $w' \neq w+e_j$, there is another position $k \in \{1, \ldots, d\}$
        besides $j$, in which $w'$ and $w$ differ.
        Assume $t=1$:
        The conditions above imply $w_j=0, w'_j=1$ and
        we can furthermore assume $w_k=0, w'_k=1$,
        since otherwise $C_{w'}C_w = 0$ according to \cref{schnitt-glb}.
        Then \cref{schnitt-glc} and \cref{schnitt-F-hilf1} imply
        \[ C_{w'} C_w F_{e_k} = C_{w'} C_w (F_0 + F_{e_k}) 
            = C_{w'}C_w \left(
                2\sum_{\substack{w'' \in \IF_2^d\\ w''_k = 0}} C_{w''}
            \right)
            = 0. 
        \]
        The case $t=0$ is proven analogously.
    \end{par}
\end{proof}

\begin{par}
    Another general proposition can be made using symmetries:
\end{par}
\begin{prop}
    \label{schnitt-lok-psi}
    \mbox{}
    \begin{enumerate}[(i)]
        \item
            There is an isomorphism of graded rings,
            \[ \psi: \KC(I^d) \xrightarrow{\sim} \KC(I^d), \]
            which is uniquely determined by
            $\psi(C_v) = C_{v+(1,\ldots, 1)}$.
            The equation
            \[ \psi(F_v) = (-1)^{\angles{v,(1,\ldots, 1)}} F_v \]
            holds.

        \item
            There is an operation of the symmetric group $S_d$ onto $\KC(I^d)$,
            which is for $\sigma \in S_d$ uniquely determined by
            \[ \cdot^\sigma: \KC(I^d) \to \KC(I^d), C_v \mapsto C_{v^\sigma}. \]
            The equation
            \[ (F_v)^\sigma = F_{v^\sigma} \]
            holds.
    \end{enumerate}
    Both automorphisms of graded rings are compatible with the local degree $\ldeg$.
\end{prop}
\begin{proof}
    The existence and uniqueness of the morphisms $\psi$ and $\cdot^\sigma$
    is evident by the definition of the combinatorial Chow ring
    written in the form of \cref{schnitt-glb}-\cref{schnitt-glc}.
    The values of $F_v$ can be calculated immediately.
    For the compatibility with $\ldeg$ it suffices to consider
    proper monomials (compare \cref{schnitt-moving}).
    Let $C_{v_0} C_{v_1} \cdot \cdots \cdot C_{v_d}$ be a non-trivial proper monomial.
    Without loss of generality we may assume that $v_0, \ldots, v_d$ form an ascending
    chain.
    Since the chains 
    $(v_0)^\sigma, \ldots, (v_d)^\sigma$ and 
    $(v_d + (1,\ldots, 1)), \ldots, (v_0 + (1, \ldots, 1))$ 
    are then ascending as well,
    we can conclude by \cref{schnitt-konk-deg}
    \[ 1= \ldeg_{\KC(I^d)}( C_{v_0} \cdot\cdots\cdot C_{v_d}) 
        = \ldeg_{\KC(I^d)}( (C_{v_0} \cdot\cdots\cdot C_{v_d})^\sigma)
        = \ldeg_{\KC(I^d)}( \psi(C_{v_0} \cdot\cdots\cdot C_{v_d})). \]
\end{proof}

\begin{kor}
    \label{schnitt-lok-ungerade}
    Let $v_0, \ldots, v_d \in \IF_2^d$ with
    $(\sum_{i=0}^d v_i)\cdot (1,\ldots, 1) = 1$.
    Then the equation
    \[ \ldeg(F_{v_0} \cdot\cdots\cdot F_{v_d})=0 \]
    holds.
\end{kor}
\begin{proof}
    According to \cref{schnitt-lok-psi} we have $\ldeg \circ \psi = \ldeg$.
    It follows
    \begin{align*}
       \ldeg(F_{v_0} \cdot\ldots\cdot F_{v_d}) & = \ldeg(\psi(F_{v_0} \cdot\ldots\cdot
        F_{v_d})) \\
        &= (-1)^{\sum_{i=0}^d v_{i}\cdot(1, \ldots, 1)} \ldeg(F_{v_0} \cdot\ldots\cdot F_{v_d})
        = - \ldeg(F_{v_0} \cdot\ldots\cdot F_{v_d}) 
    \end{align*}
    and therefore the claim.
\end{proof}


\subsection{Computation of $\KC(I^2)$ and $\KC(I^3)$} 
\label{schnitt-kalk}
\begin{par}
    With the propositions of the last section we are able to
    compute all intersection numbers in the case $d=2$ and $d=3$.
\end{par}
\begin{satz}
    \label{schnitt-zahlen-d-2}
    Let $d=2$ and $v_1, v_2, v_3$ vectors in $\IF_2^2$. 
    Then the following holds:
    \[
        \ldeg(F_{v_1}F_{v_2} F_{v_3}) = \begin{cases}
            -32 & \textrm{ if } v_1=v_2=v_3 = (1,1), \\
            16  & \textrm{ if } \{v_1,v_2,v_3\} = \{(1,0),(0,1),(1,1)\}, \\
            0   & \textrm{ otherwise}.
        \end{cases}
    \]
\end{satz}
\begin{proof}
    For brevity, we denote vectors of $\IF_2^2$ by concatenation of the digits,
    i.e., $10 = (1,0)$.
    The calculation $\ldeg(F_{10}F_{01}F_{11})=(-4)^2 = 16$ was already done in
    \cref{schnitt-F-deg}. 
    By \cref{schnitt-F-gl1} and \cref{schnitt-F-gl3} we deduce from \cref{schnitt-F}
    the equation
    \begin{equation}
        \label{schnitt-zahlen-gemischt}
        \begin{aligned}
            2 F_{10}^2 F_{11} 
            &= F_{10}^2 ((F_{11}+F_{01}) + (F_{11} - F_{01})) \\
            &= F_{10}(F_{10}-F_{00})(F_{11} + F_{01}) + F_{10}(F_{10} + F_{00})(F_{11}-F_{01})\\
            &= 0.
        \end{aligned}
    \end{equation}
    Together with \cref{schnitt-F-gl2} we infer
    \[
        F_{11}^3 = F_{11}(F_{11} - F_{00})^2 = F_{11}(F_{10} - F_{01})^2 =
        -2F_{10}F_{01}F_{11}
    \]
    and thus $\ldeg(F_{11}^3) = -32$.
    All other monomials vanish either by \cref{schnitt-lok-ungerade} 
    or equation \cref{schnitt-F-gl2} in \cref{schnitt-F}.
\end{proof}

\begin{par}
    The situation is more complicated in the case $d=3$. We introduce some 
    shortcuts in the notation:
    As in the previous proof we denote vectors of $\IF_2^3$ by concatenation of the
    digits
    such that $101$ means the vector $(1,0,1)$.
    Additionally we order the vectors from $\IF_2^3$ 
    first by the number of ones, then lexicographical:
    \[ 000 \prec 100 \prec 010 \prec 001 \prec 110 \prec 101 \prec 011 \prec 111. \]
    Four-tuples of vectors $(v_0, v_1, v_2, v_3) \in (\IF_2^3)^4$ are ordered
    lexicographical once again:
    $(v_0, \ldots v_3) \prec (v'_0, \ldots v'_3)$ holds iff there is an $N$
    such that $v_N \prec v'_N$ and $v_i = v'_i$ for all $i<N$.
\end{par}
\begin{par}
    Let the symmetric group $S_4$ operate on the four-tuples by permutation and the
    symmetric group $S_3$ operate on the four-tuples by uniformly permuting each vector
    in the tuple:
    Set for each $\sigma \in S_4, \tau \in S_3$
    \[
        (v_0, \ldots v_3)^{\sigma, \tau} := (v_{\sigma(0)}^\tau, \ldots v_{\sigma(3)}^\tau).
    \]
    According to \cref{schnitt-lok-psi} the equation
    \[
        \ldeg(F_{v_0}\cdots F_{v_3}) = \ldeg(F_{v_\sigma(0)}^\tau \cdots
        F_{v_{\sigma(3)}})
    \]
    holds  for all tuples 
    $(v_0, \ldots v_3)\in (\IF_2^3)^4$, $\sigma \in S_4$ and all $\tau \in S_3$. 
    Hence it suffices to calculate the value of $\ldeg(F_{v_0} \cdots F_{v_3})$
    only for the minimal element in the set
    \[ 
        \left\{ (v_0,v_1,v_2,v_3)^{\sigma,\tau} \mid \sigma \in S_4, \tau \in S_3
        \right\}.
    \]
\end{par}
\begin{satz}
    \label{schnitt-zahlen-d-3}
    Let $V=(v_0, \ldots v_3) \in (\IF_2^3)^4$ be a 4-tuple of vectors,
    such that 
    \[ (v_0,v_1,v_2,v_3) \leq (v_0,v_1,v_2,v_3)^{\sigma,\tau} \]
    holds for all $\sigma \in S_4, \tau \in S_3$.
    Then the intersection numbers in $\KC(I^3)_\IQ$ are
    \[
        \ldeg(F_{v_0}F_{v_1}F_{v_2}F_{v_3}) = \begin{cases}
            -64   & \textrm{ if } V=(100,010,001,111), \\
            -64   & \textrm{ if } V=(100,010,101,011), \\
            -64   & \textrm{ if } V=(100,110,101,111), \\
            128   & \textrm{ if } V=(100,011,011,111), \\
            128   & \textrm{ if } V=(100,111,111,111), \\
            128   & \textrm{ if } V=(110,110,101,011), \\
            -128  & \textrm{ if } V=(110,101,111,111), \\
            512   & \textrm{ if } V=(111,111,111,111), \\
            0    & \textrm{ otherwise }.
        \end{cases}
    \]
\end{satz}
\begin{proof}
    \begin{par}
        For brevity set 
        $d(v_0,\ldots,v_3):=\ldeg(F_{v_0} \cdots F_{v_3})$ 
        for any tuple $(v_0,\ldots v_3) \in (\IF_2^3)^4$.
        First we deduce some equations:
        In analogy to \cref{schnitt-zahlen-gemischt} we get for each $v \in \IF_2^3$
        and each unit vector $e \in E$ the relation
        \begin{equation}
            \label{schnitt-d3-e2}
            \begin{split}
                2 F_e^2 F_v & = F_e^2 ((F_v - F_{v+e}) + (F_v + F_{v+e})) \\
                & =  F_e (F_e + F_0)(F_v - F_{v+e})
                    +F_e (F_e - F_0)(F_v + F_{v+e}) \\
                & = 0.
            \end{split}
        \end{equation}
        Together with \cref{schnitt-F-gl2} one deduces for each $v \in \IF_2^3$
        the equation
        \begin{equation}
            \label{schnitt-d3-ele2}
            F_{110}^2 F_v = (F_{110} - F_{000})^2 F_v = (F_{100} - F_{010})^2 F_v
            = -2 F_{100}F_{010}F_v.
        \end{equation}
        Furthermore by combining \cref{schnitt-lok-ungerade} and
        \cref{schnitt-F-gl2} we get the equation
        \begin{equation}
            \label{schnitt-F-gl2e}
            d(e,u,v,w) = d(e, u+e, v+e, w)
        \end{equation}
        for each $u,v,w \in \IF_2^3$ and $e \in E$.
    \end{par}
    \begin{par}
        We are now able to calculate the non-trivial values of $d(v_0, \ldots v_3)$
        in ascending order according to the order $\prec$.
    \end{par}
    \paragraph{\textbf{Tuples with $\mathbf{v_0=100, v_1=010}$}} 
    \begin{par}
        According to \cref{schnitt-F-gl1} and \cref{schnitt-d3-e2} the smallest tuple
        with non-trivial intersection number has to start with
        $v_0=100, v_1=010, v_2=001$ and \cref{schnitt-F-deg} implies
        $v_3 = 111$ and the value
        $d(V) = (-4)^3 = -64$.
        For the next possible tuple $v_0=100,v_1=010,v_2=110$ 
        \cref{schnitt-F-gl2e} and \cref{schnitt-d3-e2} imply
        \[ d(100,010,110,v_3) = d(100,010,010,v_3+100) = 0\ \forall v_3 \in \IF_2^3. \]
        For the next in order, $v_0=100,v_1=010,v_2=101$ we use
        once more \cref{schnitt-F-gl2e} and get
        \[ d(100,010,101,v_3) = d(100,010,001,v_3+100). \]
        By \cref{schnitt-F-deg} the only non-trivial result is $v_3=011$, 
        thus
        \[ d(100,010,101,011) = d(100,010,001,111) = -64. \]
        Tuples of the form $v_0=100,v_1=010,v_2=101$ are not minimal under the operation of $S_4$
        and $S_3$.
    \end{par}
    \begin{par}
        It remains to examine $d(100,010,111,111)$. Once again we use
        \cref{schnitt-F-gl2e}
        and get 
        \[ d(100,010,111,111) = d(100,010,001,001) = 0. \]
        This gives a complete description of the case $v_0=100,v_1=010$.
    \end{par}
    \paragraph{\textbf{Tuples with $\mathbf{v_0=100, v_1 \succ 010}$ }} 
    \begin{par}
        We can ignore tuples with $v_1=001$, since these are not minimal under 
        the operations of $S_4, S_3$.
    \end{par}
    \begin{par}
        The next tuples in order have $v_1=110$ and these can be reduced by
        \[ d(100,110, v_2, v_3) = d(100,010, v_2+100, v_3) = d(100,010, v_2, v_3+100) \]
        to combinations already calculated.
        Thus the only non-trivial combination of this type is
        $V=(100,110,101,111)$ with
        \[ d(100,110,101,111) = d(100,010,001,111) = -64. \]
    \end{par}
    \begin{par}
        Tuples with $v_1=101$ are once again not minimal with respect to the
        $(\sigma,\tau)$-operation, so the next possible combination is 
        $V=(100,011,v_2,v_3)$. The minimality requires here 
        $011 \preceq v_2 \preceq v_3$, thus $v_2,v_3 \in \{011,111\}$.
        As \cref{schnitt-lok-ungerade} yields $d(100,011,111,111)=0$,
        we may assume $v_2=011$.
        Then one deduces by \cref{schnitt-d3-ele2}
        \[ d(100,011,011,v_3) = -2d(100,010,001,v_3) \]
        and therefore a non-trivial result with $v_3=111$ only.
    \end{par}
    \begin{par}
        The last possible tuple is $V=(100,111,111,111)$. 
        We deduce by \cref{schnitt-F-gl2e}:
        \[ d(100,111,111,111) = d(100,011,011,111) = 128.\]
    \end{par}
    \begin{par}
        So far we have covered all $(\sigma,\tau)$-minimal tuples starting with a vector
        from $E=\{100,010,001\}$.
    \end{par}
    \paragraph{\textbf{Remaining cases}}
    \begin{par}
        The next tuples in order are $V=(110,110,v_2,v_3)$. 
        These can be transformed by \cref{schnitt-d3-ele2} into
        \[ d(110,110,v_2,v_3) = -2d(100,010,v_2,v_3) \]
        and therefore yield a non-trivial result 
        only with $v_2=101, v_3 = 011$:
        \[ d(110,110,101,011) = -2d(100,010,101,011) = 128. \]
    \end{par}
    \begin{par}
        If $V$ is $(\sigma,\tau)$-minimal and contains the vector $110$ only once, 
        it also contains the vectors $101$ and $011$ only once. 
        Therefore $111$ is at least once in $V$. 
        By \cref{schnitt-lok-ungerade} we can only get a non-trivial result
        if $111$ is contained in $V$ twice or four times. 
        Hence, the only remaining cases are $V=(110,101,111,111)$ and
        $V=(111,111,111,111)$.
    \end{par}
    \begin{par}
        For the first one we deduce by \cref{schnitt-F-gl2}
        \begin{align*}
            d(110,101,111,111) & = \ldeg(F_{110}F_{101}F_{111}F_{111}) \\
            & = \ldeg\big(F_{110}F_{101}(F_{111} - F_{001})^2\big) \\
            & = \ldeg\big(F_{110}F_{101}(F_{101} - F_{011})^2\big) \\
            & = \ldeg(F_{110}F_{101}F_{011}^2 - 2 F_{110}F_{101}^2F_{011}) \\
            & = -\ldeg(F_{110}^2F_{101}F_{011}) \\
            & = -128.
        \end{align*}
        In this calculation, the already proven relations
        \[
            \ldeg(F_{110}F_{101}F_{001}^2) = 
            \ldeg(F_{110}F_{101}F_{001}F_{111}) = \ldeg(F_{100}F_{110}F_{011}F_{111}) = 0
        \]
        were used.
    \end{par}
    \begin{par}
        The case $V=(111,111,111,111)$ is shown by the analogous computation
        \begin{align*}
            d(111,111,111,111) & = \ldeg(F_{111}^4) \\
            & = \ldeg\big( F_{111}^2 (F_{111} - F_{100})^2  + 2F_{111}^3 F_{100}\big) \\
            & = \ldeg\big( F_{111}^2 (F_{110} - F_{101})^2  + 2F_{111}^3 F_{100}\big) \\
            & = -2 \ldeg( F_{110}F_{101}F_{111}^2) + 2 \ldeg( F_{100} F_{111}^3) \\
            & = -2 \cdot (-128 ) + 2 \cdot 128 \\
            & = 512.
        \end{align*}
    \end{par}
\end{proof}


\subsection{A Vanishing Condition} 

\begin{par}
    The calculations in \cref{schnitt-kalk} suggest that a lot of intersection products
    vanish. We make this observation precise in the following vanishing conjecture:
\end{par}
\begin{defn}
    Let $\Part=\{P_1, \ldots, P_l\}$ be a partition of the set $\{1, \ldots, d\}$
    and $v=(v_1, \ldots, v_d) \in \IF_2^d$.
    Then set 
    \[
        \alpha(\Part, v) := \# \{ i \in \{1, \ldots, l\} \mid \exists j \in P_i, v_j = 1 \}.
    \]
\end{defn}
\begin{defn}[vanishing condition]
    \label{limit-konv-bed}
    Let $d \in \IN$. We say that $d$ verifies the vanishing condition, 
    iff for each partition $\Part$ of $\{1, \ldots, d\}$
    and $v_0, \ldots, v_d \in \IF_2^d$ with
    \[
        \sum_{i} \alpha(\Part, v_i) < d+|\Part|
    \]
    the intersection number
    \[
        \ldeg( \prod_{i} F_{v_i} )
    \]
    vanishes.
\end{defn}
\begin{par}
    Our calculations in \cref{schnitt-kalk} verify the vanishing condition in two cases:
\end{par}

\begin{kor}
    For $d=2$ and $d=3$ the vanishing condition \cref{limit-konv-bed} is satisfied.
\end{kor}
\begin{proof}
    \begin{par}
        For the partition $\Part=\{\{1, \ldots, d\}\}$ the condition \cref{limit-konv-bed}
        is always true, since
        $\ldeg(F_0 \cdot \prod_{i=1}^d F_{v_i}) = 0$ holds for all
        $v_1, \ldots, v_d \in \IF_2^d$.
    \end{par}
    \begin{par}
        If $d=2$, there is only the partition 
        $\Part=\{\{1\},\{2\}\}$ left.
        For each $v_1, \ldots, v_3 \in \IF_2^d$ 
        with $\ldeg(F_{v_1}F_{v_2}F_{v_3}) \neq 0$
        we have by \cref{schnitt-zahlen-d-2}
        $\sum_{i=1}^3 |v_i| \geq 4$.
        Therefore \cref{limit-konv-bed} is true.
    \end{par}
    \begin{par}
        In the case $d=3$ it is easy to check in all non-trivial combinations of
        \cref{schnitt-zahlen-d-3} that 
        \[
            \sum_{i} \alpha(\Part, v_i) \geq d+|\Part|
        \]
        is satisfied.
    \end{par}
\end{proof}
With the use of the computer algebra system Sage \cite{sage}
we are also able to verify the vanishing condition 
for $d=4$ and $d=5$. Therefore we conjecture:
\begin{conj}
    The vanishing condition is true for arbitrary $d \in \IN$.
\end{conj}


    \appendix

\section{The Category of Simplicial Sets}
\label{sk-kap}

\begin{par}
    A good description for the simplicial structure of the special fibre in our setting
    is given by simplicial sets. For convenience we repeat the basic definitions
    and properties used in this paper.
\end{par}
\begin{defn}
    A partially ordered set is a set $A$ endowed with a reflexive transitive and
    antisymmetric relation $\leq$. 
    A morphism of partially ordered sets $f: A \to B$ is a map of sets, 
    which is monotonically increasing. 
    This means $a \leq a'$ implies $f(a) \leq f(a')$ for each $a,a' \in A$.
\end{defn}
\begin{bem}
    The category of partially ordered sets has finite products. The product of two 
    partially ordered sets is given by the cartesian product $A \times B$
    endowed with the product order
    \[ 
        (a,b) \leq (a',b')  \quad \iff 
        \quad a \leq a' \text{ and } b \leq b'
    \]
\end{bem}
\begin{defn}
    Let $\Delta$ denote the simplicial category. This has as objects the finite ordered sets
    $[n]:=\{0, \ldots, n\}$ for each natural number $n \in \IN$ 
    and monotonically increasing mappings as morphisms.
    A simplicial set $\RK$ is a contravariant functor
    $\RK: \Delta \to \set$. For each $n \in \IN_0$ the set $\RK([n])$ is denoted by
    $\RK_n$ and its elements are called $n$-simplizes.
    A morphism of simplicial sets is a strict transformation of
    functors.
    The category defined by this is called \emph{category of simplicial sets} and denoted
    by $\sset$.
\end{defn}
\begin{defn}
    \label{sk-standard}
    The functor $\Delta[n]:=\Hom_{\Delta}(\cdot, [n])$ is a simplicial set,
    the \emph{standard $n$-simplex}.
\end{defn}
\begin{defn}
    Let $\RK$ denote a simplicial set and $k \in \IN$. A $k$-simplex $\sigma \in \RK_k$ is
    called degenerate if there exists a morphism $d: [k] \to [k-1]$, such that $\sigma$
    lies in the image of $d^*: \RK_{k-1} \to \RK_k$.
    The set of all nondegenerate simplices is denoted by $\RK_k^{\textrm{nd}}$.
\end{defn}
\begin{bem}
    \label{sk-sk-prod}
    By standard arguments of category theory (see \cite[Prop 8.7, Cor 8.9]{awodey}) the category $\sset$ has limits and
    colimits, which can be constructed component-by-component.
\end{bem}
\begin{par}
    In particular one has the following description of products of standard-1-simplizes:
\end{par}
\begin{kor}
    \label{sk-prod-poset}
    For each $d \in \IN$ there is a canonical bijection
    \[
        (\Delta[1])^d \simeq \Hom_{\poset}(\cdot, [1]^d),
    \]
    where $[1]^d$ is seen as product of partially ordered sets.
\end{kor}
\begin{proof}
    Since the product $(\Delta[1]^d)$ can be constructed component-by-component
    we get for each $n \in \IN_0$ a functorial isomorphism
    \[
        (\Delta[1])^d_n 
        \simeq \prod_{d} \Hom_{\Delta}([n], [1])
        \simeq \Hom_{\poset}([n], [1]^d)
    \]
    as claimed.
\end{proof}
\begin{par}
    Furthermore each simplicial set is a colimit of standard simplicial sets:
\end{par}
\begin{prop}
    \label{sk-kolim-simpl}
    For each simplicial set $\RK \in \sset$ the equations
    \[ K_n \simeq \Hom_{\sset}(\Delta[n], K_\cdot) \]
    and
    \[ K_\cdot = \colim_{\Delta K_\cdot} \Delta[n] = \colim_{\Delta' K_\cdot} \Delta[n]
    \]
    hold.
\end{prop}
\begin{proof}
    This standard fact is proven for example 
    in \cite[Lemma 3.1.3, Lemma 3.1.4]{hovey}.
\end{proof}
\begin{defn}
    \label{sk-einfach-kompl}
    For each $i \in \{0,\ldots,k\}$, let $s_i$ denote the morphism
    \[ s_i: [0] \to [k], 0 \mapsto i. \]
    A simplicial set $\RK$ is called \emph{simplicial set without multiple simplices},
    if the map
    \[ \varphi: \coprod_{k=0}^\infty K^{\mathrm{nd}}_k \to \mathcal{P}(K_0), 
        t\in K^{\mathrm{nd}}_k \mapsto \{K(s_0)(t), \ldots, K(s_k)(t)\} \]
    is a monomorphism.
    If this is true, we denote the image of $\varphi$ by
    \[
        \RK_S := \im(\varphi) \subseteq \mathcal{P}(K_0).
    \]
\end{defn}
\begin{par}
    We mostly consider only simplicial sets of this type, 
    since the morphisms are uniquely defined by the mapping of the vertices:
\end{par}
\begin{prop}
    \label{sk-morph-mfs-eind}
    Let $\RK$ and $\RK'$ be two simplicial sets, 
    where $\RK'$ has no multiple simplices 
    and $f,f': \RK \to \RK'$ two morphisms of simplicial sets.
    If the restriction of $f$ and $f'$ onto the $0$-simplices,
    \[
        f\mid_0, f'\mid_0: K_0 \to K'_0,
    \]
    agree, then $f=f'$ holds.
\end{prop}
\begin{proof}
    It is enough to show that for each $k \in \IN$ and each nondegenerated $k$-simplex
    $\sigma \in \RK_k$ the equation $f(\sigma) = f'(\sigma)$ holds.
    Since $\RK$ has no multiple simplices, 
    the $k$-simplices $f(\sigma)$ and $f'(\sigma)$ are uniquely determined by
    $\varphi(f(\sigma))$ and $\varphi(f'(\sigma))$.
    Since these elements depend only on $f_0$ and $f'_0$,
    the proposition is true.
\end{proof}
\begin{par}
    For the description of ramified base-change we need a subdivision of simplicial sets.
    This can also described completely categorial (see \cite[Appendix 1]{segal}):
\end{par}
\begin{defn}
    \mbox{}
    \label{sk-unt-def}
    \begin{enumerate}[(i)]
        \item
            Let $k \in \IN$. We denote by $\tilde \unt_k$ the functor
            \[
                \tilde\unt_k: \Delta \to \Delta 
            \]
            given on objects by
            \[ [n] \mapsto [(n+1)\cdot k - 1] \]
            and on morphisms by
            \[ \Hom_\sset([n],[m]) \ni \varphi \mapsto 
                \left( ak+b \mapsto ak+\varphi(b)
                    \textrm{ for } 0 \leq b < k
                \right).
            \]
        \item
            The functor induced by $\tilde\unt_k$
            \[ \unt_k: \sset \to \sset, K_\cdot \mapsto \tilde\unt_k \circ K_\cdot \]
            is called the \emph{$k$-fold subdivison functor}.
    \end{enumerate}
\end{defn}

    \bibliography{literatur}
\end{document}